\documentclass[]{scrartcl}
\usepackage[a4paper,top=27mm,bottom=20mm,inner=25mm,outer=20mm]{geometry}

\usepackage[utf8]{luainputenc}
\usepackage[USenglish]{babel}
\usepackage{csquotes}

\usepackage[%
  backend=bibtex,bibencoding=ascii,
  style=numeric-comp,
  giveninits=true, uniquename=init, %abbreviate first names
  natbib=true,
  url=false,
  doi=true,
  isbn=false,
  backref=false,
  maxnames=99,
  ]{biblatex}
\usepackage{amsmath}
\usepackage{amssymb}
\usepackage{mathtools}
\usepackage{pgfplots}
\pgfplotsset{compat=1.18}
\usepackage{amsthm}
\usepackage{placeins}
\newtheorem{theorem}{Theorem}
\newtheorem{lemma}[theorem]{Lemma}

\newtheorem{remark}[theorem]{Remark}
\newtheorem{proposition}[theorem]{Proposition}

\usepackage[plainpages=false,pdfpagelabels,hidelinks,unicode]{hyperref}

\usepackage{ifluatex}
\ifluatex
  \usepackage[no-math]{fontspec}
\else
  \usepackage[T1]{fontenc}
\fi
\usepackage{newpxtext,newpxmath}

\usepackage{color}
\usepackage{graphicx}

\definecolor{col1}{RGB}{         0  135.4688  255.0000}
\definecolor{col2}{RGB}{         0    9.9609  255.0000}
\definecolor{col3}{RGB}{  115.5469         0  255.0000}
\definecolor{col4}{RGB}{  241.0547         0  255.0000}
\definecolor{col5}{RGB}{  255.0000         0  143.4375}
\definecolor{col6}{RGB}{  255.0000         0   17.9297}

\pgfplotscreateplotcyclelist{hierarchy6}{%
	{col1,thick,mark=o},
	{col2,thick,mark=pentagon},
	{col3,thick,mark=square},
	{col4,thick,mark=triangle},
	{col5,thick,mark=x},
   {col6,thick,mark=otimes}}

\usepackage{setspace}
\usepgfplotslibrary{external}
\newcommand{\E}{\mathcal{E}}    % Cahn-Hilliard Energy
\newcommand{\Eh}{\mathcal{E}_H} % Cahn-Hilliard-Hyperbolic Energy
\newcommand{\Ehd}{\mathcal{E}_{H,d}} % Cahn-Hilliard-Hyperbolic Energy

\renewcommand{\O}{\mathcal{O}}

\newcommand{\FF}{\mathcal{F}_{\eps}}
\renewcommand{\S}{\mathcal{S}_{\eps}}
\newcommand{\ww}{\mathbf{w}}
\newcommand{\qq}{\mathbf{q}}
\newcommand{\uu}{\vec{u}}
\newcommand{\pp}{\mathbf{p}}
\newcommand{\xx}{\mathbf{x}}
\newcommand{\cc}{\varsigma}
\newcommand{\Nx}{{N}_x}
\newcommand{\ndof}{\mathtt{ndof}}

\newcommand{\R}{\mathbb{R}}
\newcommand{\N}{\mathbb{N}}
\newcommand{\Te}{T_{\text{end}}}
\newcommand{\eps}{\varepsilon}
\newcommand{\dx}{\Delta x}
\newcommand{\dt}{\Delta t}

\DeclareMathOperator{\sign}{sign}
\DeclareMathOperator{\rand}{rand}
\DeclareMathOperator{\Id}{Id}

\newcommand{\orcid}[1]{ORCID:~\href{https://orcid.org/#1}{#1}}

\title{
    Justification and structure- and asymptotic-preserving discretizations of a hyperbolized Cahn-Hilliard equation
}

\author{
    Jan~Giesselmann\thanks{\orcid{0009-0008-0217-7244}, Department of Mathematics, Technical University of Darmstadt, Dolivostr 15, 64293, Germany},
    Fabio~Leotta\thanks{\orcid{0000-0001-6131-0107}, Department of Mathematics, Technical University of Darmstadt, Dolivostr 15, 64293, Germany},
    Hendrik~Ranocha\thanks{\orcid{0000-0002-3456-2277}, Institute of Mathematics, Johannes Gutenberg University Mainz, Staudingerweg 9, 55128 Mainz, Germany}, and
    Jochen Sch\"utz\thanks{\orcid{0000-0002-6355-9130}, Faculty of Sciences \& Data Science Institute, Hasselt University, Agoralaan Gebouw D, BE-3590 Diepenbeek, Belgium}
}
\date{\today} % TODO

\begin{document}

\maketitle

\begin{abstract}
  We study a hyperbolic approximation ("hyperbolization") of the Cahn-Hilliard (CH) equation, originally proposed by Dhaouadi, Dumbser, and Gavrilyuk (2025, DOI: \href{https://doi.org/10.1098/rspa.2024.0606}{10.1098/rspa.2024.0606}) and study its convergence towards the CH model in a relaxation limit both via formal asymptotic expansions and, for a slightly modified approximation, via the relative energy framework. Moreover, we develop energy-stable semidiscretizations of the CH equation and of this hyperbolization using upwind summation-by-parts operators in space. Subsequently, we combine them with (additive) implicit-explicit (IMEX) Runge-Kutta methods based on a convex-concave splitting. We show that the resulting method is asymptotic preserving, i.e., it converges in the limit of the relaxation parameter to a stable discretization of the original CH equation.
The choice of the necessary parameters is guided by the a~priori error estimate based on the relative energy framework.

\end{abstract}

\section{Introduction}
We propose and analyze an alternative discretization for the Cahn-Hilliard equation \cite{1958_Cahn_article} through hyperbolization.
The Cahn-Hilliard equation describes two-phase flow, where the phases are not separated by a sharp interface but a diffuse (``smeared-out'') one.
Given some spatial domain $\Omega \subset \R^d$, the Cahn-Hilliard equation \cite{1958_Cahn_article} with constant mobility is given by
\begin{align}
    \label{eq:CH}
    \partial_t c = \Delta (g'(c)-\gamma \Delta c) \qquad
    \text{in } \Omega \times \R^{+},
    %\forall (x, t) \in \Omega \times \R^{+},
\end{align}
where $\gamma > 0$ is a small parameter that determines the interface thickness, and the function $g: \R \rightarrow \R$ is the so-called double-well potential. Here, we use the well-established choice
\begin{align}
 \label{eq:doublewell}
 g(c) := \frac 1 4 \left(c^2-1\right)^2;
\end{align}
other choices are possible. For a thorough recent review on the Cahn-Hilliard equation, we refer to \cite{2022_Review_CahnHilliard}.
The Cahn-Hilliard equation comes with an associated energy $\E$ given by
\begin{align}
 \label{eq:E}
 \E(c) = \int_{\Omega} \left( g(c) + \frac{\gamma}{2} |\nabla c|^2 \right) \mathrm{d} x.
\end{align}
For smooth solutions $c$ to \eqref{eq:CH} with appropriate boundary conditions\footnote{Throughout this work, we use periodic boundary conditions for the ease of presentation. For homogeneous Neumann boundary conditions on both $c$ and the chemical potential $\mu:=g'(c) - \gamma \Delta c$, the result is equally true.}, we can use integration by parts to show the well-known energy dissipation property
\begin{align*}
  \frac{\mathrm d}{\mathrm dt} \E(c) =
  %\int_{\Omega} \left(g'(c) - \gamma \Delta c\right) \Delta \left(g'(c)-\gamma \Delta c\right) \mathrm{d}x \leq 0.
  - \int_{\Omega} \left| \nabla \left(g'(c) - \gamma \Delta c\right)\right|^2 \mathrm{d}x \leq 0.
\end{align*}
Hence, the energy must not increase, which is, together with mass conservation of $c$, of utmost physical importance. Numerical schemes should be designed in such a way that they respect these physical properties.

In this work, we propose a way to solve \eqref{eq:CH} through generalizing an existing hyperbolization of Dhaouadi, Dumbser, and Gavrilyuk \cite{dhaouadi2025first}.
Hyperbolization means that instead of solving \eqref{eq:CH}, a first-order \emph{hyperbolic} system of equations
\begin{align}
 \label{eq:CHH}
 \ww_t + \nabla \cdot \FF(\ww) &= \S(\ww) \qquad \text{in }\Omega \times \R^{+}
\end{align}
is being solved. Here, both $\FF$ and $\S$ depend on a ``small'' parameter $\eps > 0$, which makes $\ww$ also depend on $\eps$ (suppressed for the ease of notation). The system in \eqref{eq:CHH} is designed in such a way that for $\eps \rightarrow 0$, the first component of $\ww$ formally converges to a solution $c$ of \eqref{eq:CH}.
In \cite{dhaouadi2025first}, the hyperbolization is constructed such that  there is an energy $\Eh$ associated to \eqref{eq:CHH}, 
details will be presented in the next section.
By energy we refer to a functional $\Eh(\ww)$ that decays over time
and formally, for $\eps \rightarrow 0$, %$\Eh(\ww)$
converges to $\E(c)$.
In this work, we present spatial discretizations based on (upwind) summation-by-parts (SBP) operators \cite{Mattsson2017} for both \eqref{eq:CH} and \eqref{eq:CHH} that preserve the decay of the respective energies.
Subsequently, we show that these discretizations, in combination with a suitable IMEX (implicit/explicit) time integration, will for $\eps \rightarrow 0$ converge to each other, i.e., the algorithm for \eqref{eq:CHH} formally converges to the algorithm for \eqref{eq:CH} and the underlying scheme is hence asymptotic preserving \cite{Jin2012}.

The novelties of this work are:
\begin{itemize}
 \item We adapt the hyperbolization of \cite{dhaouadi2025first} by introducing a small parameter $\eps$ and analyze its convergence properties with respect to $\eps$.
 \item We provide an a~priori error estimate, quantifying the convergence of a solution of a specific instance of \eqref{eq:CHH} to a solution of \eqref{eq:CH} as long as the latter admits a strong solution, based on a relative energy framework for a slightly modified variant of this hyperbolization.
 \item We present energy-stable spatial algorithms for \eqref{eq:CH} and \eqref{eq:CHH}, and show that they are close to each other for small $\eps$, i.e., we show that the discretization of \eqref{eq:CHH} is asymptotic preserving.
\end{itemize}

\paragraph{Related literature}

Since the seminal work of Eyre \cite{1998Eyre}, it is known that implicit/explicit (IMEX) time integration, see, e.g., \cite{Ascher1995,Ascher1997,Bo07,christopher2001additive,RuBosc09,2025BoscarinoPareschiRusso},
should be used to obtain energy-stable discretizations of the Cahn-Hilliard equation that do not have to rely on a severe time step restriction. (Very recently, however, there has been some criticism of this splitting, at least in the context of the Allen-Cahn equation \cite{2026Dondl}.)
Building upon Eyre's splitting, vast literature is available, see, e.g., \cite{2010ShenYang,2015HanWang,christlieb2013unconditionally,2016GuoFilbertXu,2024KirkRiviereMasri,2026IslamSinha} and the references therein.
All relevant spatial discretization schemes have been applied to Eq.~\eqref{eq:CH}, in particular, the continuous Galerkin method \cite{1989Elliott,2016DiegelWangWise}, the finite volume method \cite{2008CuetoPeraire}, and (Fourier) spectral methods \cite{2007HeLiuTang}, to mention only a few highly cited contributions.
The classical local discontinuous Galerkin (DG) method for Cahn-Hilliard has been introduced in \cite{xia2007local}, see also \cite{2016GuoFilbertXu,2020FrankRuppKuzmin,manzanero2020free,2024KirkRiviereMasri,2026IslamSinha} and the references therein for extensions and analysis.
In particular, \cite{2024KirkRiviereMasri} gives a hybridized DG scheme that is provably unconditionally well-posed and energy stable.

We construct provably energy-stable spatial discretizations using the general framework of SBP operators.
Classical references and an introduction to the general concept can be found in the review articles \cite{svard2014review,fernandez2014review}.
While we focus on finite differences (FDs) \cite{kreiss1974finite,strand1994summation,carpenter1994time} and discontinuous Galerkin~(DG) methods \cite{gassner2013skew,carpenter2014entropy} in the numerical experiments, the results extend directly to all other classes of schemes that fit within the SBP framework, including
finite volumes \cite{nordstrom2001finite},
continuous finite elements \cite{hicken2016multidimensional,hicken2020entropy,abgrall2020analysisI},
flux reconstruction \cite{ranocha2016summation},
active flux methods \cite{barsukow2025stability},
meshless schemes~\cite{hicken2025constructing,kwan2025robust},
and cut-cell methods~\cite{taylor2025entropy,petri2026kinetic}.
In the context of DG methods for parabolic problems, the classical notion of SBP operators corresponds to the Bassi-Rebay~1 (BR1) method \cite{bassi1997high}, which has been shown to be provably energy-stable for the Cahn-Hilliard equation in \cite{manzanero2020free}.
However, it is well-known that local DG (LDG) methods can be advantageous, e.g., in terms of stability properties for elliptic problems \cite{ABCM}.
They can be included within the SBP framework \cite{RanochaConservative2021,ortleb2025stability} using the generalized notion of upwind SBP operators \cite{Mattsson2017}, a special case of dual-pair derivative operators \cite{dovgilovich2015high}.
Upwind SBP operators can be interpreted as central SBP operators plus artificial dissipation \cite{svard2005steady,mattsson2007high}, removing nullspace consistency issues related to instability of the operator for elliptic problems.

Hyperbolizations have already been introduced several decades ago to avoid infinite speed of propagation in physical models, e.g., for the heat equation \cite{cattaneo1958forme,vernotte1958paradoxes}.
The interest in hyperbolizations has been renewed in recent years, see, e.g., \cite{antuono2009dispersive,favrie2017rapid,Ketcheson_2025} and the references therein.
Hyperbolizations can lead to a better understanding of the sometimes very difficult underlying equations, for example with respect to boundary conditions \cite{besse2022perfectly}.
Eventually, hyperbolization may also facilitate GPU implementations \cite{escalante2019efficient,wittenstein2026gpu} and lead to more versatile and more efficient numerical schemes \cite{busto2021high,toro2014advection,ranocha2025structure}, as numerical methods for hyperbolic equations have become very mature even in difficult situations such as those that occur, e.g., for massive mesh refinement or for solutions with discontinuities or very strong gradients \cite{guermond2022well,guermond2022hyperbolic}.
Moreover, the auxiliary variables introduced in hyperbolizations to approximate spatial derivatives will typically converge with the same order as the primary variable, i.e., with a higher order than most traditional discretizations of steady-state problems would yield \cite{mazaheri2016first,chamarthi2019first,schlottkelakemper2021purely}.
While one can argue that this comes at the cost of more unknowns, the additional variables can sometimes be eliminated by a careful choice of discretizations \cite{ranocha2023discontinuous}.
While the resulting discretizations can also be motivated and analyzed using traditional techniques \cite{cockburn2009superconvergent}, hyperbolizations provide an intuitive and powerful way for designing and analyzing numerical schemes.
Concerning the Cahn-Hilliard equation, we focus on the hyperbolization proposed by Dhaouadi, Dumbser, and Gavrilyuk \cite{dhaouadi2025first}.
Keim, Konan, and Rohde \cite{keim2024note} proposed another hyperbolic-elliptic approximation of the Cahn-Hilliard equation, where a first-order hyperbolic system is coupled with an elliptic equation.
Moreover, \cite{yu2026first} studied a hyperbolization of a Cahn-Hilliard system and first-order energy-stable methods.

Solving \eqref{eq:CHH} instead of \eqref{eq:CH} means that the solution to the latter equation is approximated by $\eps$-dependent solutions of the former equation.
This process can only be reasonable if, as $\eps \rightarrow 0$, the solutions converge to each other, not only on the continuous level, but also the discretized version of \eqref{eq:CHH} should converge to a discretized version of \eqref{eq:CH}.
In a broader context, such a property is called asymptotic preserving (AP) \cite{Jin99,Jin2012}. The literature on AP schemes is vast, ranging from relaxation problems, see, e.g., \cite{filbet2010class,DiPar2013,2014DimarcoPareschi,2024BoscarinoRusso} and the references therein, to more general operator-split equations, see, e.g., \cite{Kl95,dimarco2017study,BispenIMEXSWE,BoRuSca18,DeTa,HaJiLi12,Oberwolfach19,RSIMEXFullEuler,2025AnandanEtAl}.
We refer to \cite{2022Shi} for a recent review on AP schemes for multiscale physical problems.
In the context of hyperbolization, the AP concept has been applied to the Korteweg-de-Vries and Benjamin-Bona-Mahony equations \cite{bleecke2025asymptoticpreservingenergyconservingmethodshyperbolic,BiswasEtAl2025}.

While hyperbolizations have been proposed and used in many works, there are only very few rigorous results on the convergence of the hyperbolized system to the original system \cite{duchene2019rigorous}.
We use the relative entropy (also known as relative energy) method for this purpose, which has a long history for both hyperbolic and parabolic problems.
It leverages an energy or entropy structure of the considered system in order to control the distance between two solutions.
In this paper, we will use the term relative energy consistently (although we could call it relative entropy equally well).
An overview on the use of relative energy in parabolic models can be found in \cite{Juengel2016}.
For hyperbolic models, it originates in works of Dafermos \cite{Dafermos1979} and DiPerna \cite{Diperna1979}, where it was, in particular, used to establish weak-strong uniqueness.
In addition, it is frequently used to study the relationship between solutions of different models, where usually one can be understood as the relaxation limit of the other, e.g., in large friction limits of hyperbolic systems, see \cite{Tzavaras2005,Lattanzio2017,GiesselmannLattanzioTzavaras2017,Egger2023,Gallenmueller2024}.
For an extension to hyperbolic-parabolic systems, see \cite{Christoforou2018}.
A general property in this analysis is that the solution to the limiting system needs to be a strong solution whereas the solutions to the approximating system can be allowed to be (weak) entropy solutions.
Recently, it has been used to investigate the convergence of hyperbolic approximations of higher-order PDEs \cite{giesselmann2025convergence} also in the context of Cahn-Hilliard equations \cite{giesselmann2026justificationrelaxationapproximationnavierstokescahnhilliard}, where %a conceptually very different hyperbolization
a hyperbolic-elliptic approximation of the Cahn-Hilliard equation, coupled to incompressible Navier-Stokes equations, suggested in \cite{keim2024note} was studied.

\paragraph{Structure of this article}

This work is structured as follows:
In Sec.~\ref{sec:hyperbolization}, we review the hyperbolization proposed by \cite{dhaouadi2025first} and analyze it with respect to convergence properties as the hyperbolization parameter $\eps \rightarrow 0$.
Subsequently, in Sec.~\ref{sec:relativeenergyanalysis}, we provide an a~priori error analysis framework for a slightly modified hyperbolized relaxation equation.
In Sec.~\ref{sec:spatial}, we use SBP operators as spatial discretization tools to devise provably energy-stable semi-discrete schemes for both the Cahn-Hilliard equation and its relaxation.
These schemes are then discretized in time in Sec.~\ref{sec:time}, and we show that as $\eps \rightarrow 0$, the discretization of the hyperbolic relaxation system converges to a discretization of \eqref{eq:CH}.
Sec.~\ref{sec:numres} presents numerical results; Sec.~\ref{sec:conout} offers conclusions and an outlook.

\section{Asymptotic analysis of the hyperbolization}
\label{sec:hyperbolization}
Dhaouadi, Dumbser, and Gavrilyuk \cite{dhaouadi2025first} introduced a hyperbolization for \eqref{eq:CH} using the unknowns
\begin{align}
 \label{eq:ww}
  \ww = (\cc, \qq, w, \pp, \varphi)
\end{align}
through
\begin{subequations}\label{eq:CHHExplicit}
\begin{alignat}{2}
  \label{eq:CHHE1}    \cc_t     &+ \nabla \cdot \left(\frac \qq {\kappa_3} \right)              &\ = \ & 0, \\
  \label{eq:CHHE2}  \qq_t   &+ \nabla \left(g'(\cc) + \frac{\cc-\varphi}{\kappa_1}\right)           &\ = \ & -\frac \qq {\kappa_3}, \\
  \label{eq:CHHE3}  w_t     &- \nabla \cdot (\gamma \pp)                                &\ = \ & \frac{\cc - \varphi}{\kappa_1}, \\
  \label{eq:CHHE4}  \pp_t   &-\frac {\nabla w} {\kappa_2} &\ = \ & 0, \\
  \label{eq:CHHE5}  \varphi_t &                                                         &\ = \ & \frac w {\kappa_2},
\end{alignat}
\end{subequations}
where $\kappa_i$, $1 \leq i \leq 3$, are fixed parameters\footnote{Please note that in the work \cite{dhaouadi2025first}, the constants have different names; they use the symbols $\alpha$, $\beta$ and $\tau$. In their notation, $\kappa_1 = \frac 1 \alpha$, $\kappa_2 = \beta$ and $\kappa_3 = \tau$. We renamed the parameters to avoid overly heavy notation and only use parameters that asymptotically go to zero.} which should vanish to obtain the original equation \eqref{eq:CH}.
More precisely, in \cite[Prop.~1]{dhaouadi2025first}, the parameters are chosen as $\kappa_1 = \gamma$, $\kappa_2 = \gamma^3$ and $\kappa_3 = \gamma^2$, and the authors show that the solution $\cc$ of \eqref{eq:CHHExplicit} is close to a solution of \eqref{eq:CH} up to order $\gamma^2$ in a certain sense.
This obviously assumes that $\gamma$ is small, which is reasonable in many applications where the Cahn-Hilliard equation models a diffuse interface that, for $\gamma \rightarrow 0$, turns into a sharp interface \cite{1989Pego}.
In this work, in contrast, we consider $\gamma$ to be a fixed parameter different from zero, and seek hyperbolizations depending on a parameter $\eps$ that converge to \eqref{eq:CH} as $\eps$ goes to zero.

\begin{theorem}\label{thm:aphyp}
 Define the hyperbolization parameters in \eqref{eq:CHHExplicit} as $\kappa_1 = \eps$, $\kappa_2 = \gamma \eps^{k_2}$ and $\kappa_3 = \eps^{k_3}$ for positive values of $k_2, k_3 \in \N$.
 Let $\ww$ from \eqref{eq:ww} be a smooth solution that possesses a Hilbert expansion, i.e., the function $\ww$ can be written as
 \begin{align*}
  \ww(\xx,t) = \ww_0(\xx,t) + \eps \ww_1(\xx,t) + \eps^2 \ww_2(\xx,t) + \cdots.
 \end{align*}
 We assume that the initial condition on $\ww$ is well-prepared, which for our purposes means that
 \begin{align*}
  \cc(\xx, t = 0) = c(\xx, t = 0) + \O(\eps), \qquad \pp(\xx, t = 0) = \nabla \varphi(\xx, t = 0) + \O(\eps).
 \end{align*}
 Then, $\cc_0$, the first component of $\ww_0$, satisfies \eqref{eq:CH}; hence, it is a solution of the Cahn-Hilliard equation.
\end{theorem}
\begin{remark}
	The $\gamma$-scaling of $\kappa_2$ will become clear from the a~priori error analysis, see Thm.~\ref{thm:error_estimate} below.
\end{remark}
\begin{proof}
 The last equation \eqref{eq:CHHE5} implies that
 \begin{align}
  \label{eq:wkglnull}
  w_0 = \cdots = w_{k_2-1} = 0 \quad \text{and} \quad \partial_t \varphi_0 = \frac 1 \gamma w_{k_2}.
 \end{align}
 Using this in \eqref{eq:CHHE4}, we obtain
 \begin{align*}
  \partial_t \pp_0 = \frac 1 \gamma \nabla w_{k_2},
 \end{align*}
 which results in
 \begin{align*}
   \partial_t \pp_0 = \partial_t (\nabla \varphi_0).
 \end{align*}
 Together with the well-prepared initial conditions, we can hence conclude that
 \begin{align*}
  \pp_0 = \nabla \varphi_0.
 \end{align*}
 The leading-order term of Eq.~\eqref{eq:CHHE3} yields
 \begin{align*}
  \varphi_0 = \cc_0,
 \end{align*}
 and hence
 \begin{align}
    \pp_0 = \nabla \cc_0. \label{eq:ppglcc}
 \end{align}
 Eq.~\eqref{eq:CHHE2} implies that (note that if $k_3 = 1$, then this equation is void)
 \begin{align*}
  \qq_{i} = 0, \quad 0 \leq i \leq k_3 - 2.
 \end{align*}
 We then obtain to leading order that
 \begin{align}
  \label{eq:qqk3glnull}
  \nabla (\cc_0 - \varphi_0) = -\qq_{k_3-1},
 \end{align}
 which implies (with $\varphi_0 = \cc_0$) that $\qq_{k_3-1} = 0$.
 From Eq.~\eqref{eq:CHHE2}, we can then conclude that
 \begin{align*}
  \nabla \left(g'(\cc_0) + \cc_1 - \varphi_1\right) &= -\qq_{k_3}.
 \end{align*}
 Together with the next-to-leading-order term of \eqref{eq:CHHE3}, which is
 \begin{align*}
  \cc_1 - \varphi_1 = -\nabla \cdot(\gamma \pp_0) = -\nabla \cdot(\gamma \nabla \varphi_0) = -\gamma \Delta \varphi_0 = -\gamma \Delta \cc_0,
 \end{align*}
 this results in
 \begin{align}
   \label{eq:qk3}
   \nabla \left(g'(\cc_0) -\gamma \Delta \cc_0\right) &= -\qq_{k_3}.
 \end{align}
 Finally, we can plug this expression for $\qq_{k_3}$ into the leading-order terms of \eqref{eq:CHHE1} to obtain the final result
 \begin{align*}
  (\cc_0)_t - \nabla \cdot \left( \nabla \left(g'(\cc_0) -\gamma \Delta \cc_0\right) \right) = 0,
 \end{align*}
 which is obviously equivalent to \eqref{eq:CH}. Together with the initial conditions for $\cc$, this implies that $\cc_0 = c$, where $c$ is the solution to \eqref{eq:CH}.
\end{proof}
\begin{remark}
 The well-prepared initial conditions in Thm.~\ref{thm:aphyp} are the minimal requirements to prove the theorem. From the proof, it seems natural to also set
 \begin{align*}
  \varphi(\xx, t = 0) = c(\xx, t = 0), \quad w(\xx, t = 0) = 0, \quad \qq(\xx, t = 0) = 0.
 \end{align*}
\end{remark}

Thm.~\ref{thm:aphyp} gives us some flexibility in choosing the parameters $k_2, k_3 \in \N$ in $\kappa_2 = \gamma \eps^{k_2}$ and $\kappa_3 = \eps^{k_3}$ in \eqref{eq:CHHExplicit}.
For illustration purposes, we consider the one-dimensional case for a moment, i.e., $\Omega \subset \R$.
Then, using the parameters defined in Thm.~\ref{thm:aphyp}, the wave speeds of the hyperbolic equation \eqref{eq:CHH}, i.e., the eigenvalues of the Jacobian of $\FF$, are given by (compare \cite[Sec.~2(e)]{dhaouadi2025first})
\begin{align}
 \label{eq:wavespeeds}
 \lambda_1 = -\frac{\sqrt{g''(\cc)+1/\eps}}{\eps^{k_3/2}}, \;
 \lambda_2 = -\frac{1}{\eps^{k_2/2}}, \;
 \lambda_{3} = 0, \;
 \lambda_4 = \frac{1}{\eps^{k_2/2}}, \;
 \lambda_5 = \frac{\sqrt{g''(\cc)+1/\eps}}{\eps^{k_3/2}}.
\end{align}
This implies that the largest eigenvalue that will ultimately determine a CFL condition when using an explicit time discretization is of order $\max\{\eps^{-1/2 -k_3/2}, \eps^{-k_2/2}\}$.
Hence, it seems appropriate to minimize this quantity and to choose $k_3 = 1$ and $k_2 \leq 2$, resulting in $\lambda_i = \O(\eps^{-1})$, which is a frequently obtained magnitude of the wave speeds in such a context, see, e.g., \cite{bleecke2025asymptoticpreservingenergyconservingmethodshyperbolic}.
Choosing $k_2 = 2$ and $k_3 = 1$ leads to wave speeds that are balanced at order $\eps^{-1}$, while $k_2 = k_3 = 1$ makes at least the absolute values of $\lambda_2$ and $\lambda_4$ the smallest.
In the numerical results section, we investigate the influence of the parameters $k_2$ and $k_3$ on the quality of the hyperbolic solution in comparison to the Cahn-Hilliard solution.

As already mentioned, the system \eqref{eq:CHHExplicit} comes with an energy that decays in time. Using the choice of parameters from Thm.~\ref{thm:aphyp}, this energy is given by \cite[Prop.~2]{dhaouadi2025first}
\begin{align}
  \label{eq:Eh}
  \Eh(\ww) = \int_{\Omega} \left( g(\cc) + \frac \gamma 2 |\pp|^2 + \frac {(\cc-\varphi)^2}{2\eps}  + \frac {w^2} {2\gamma \eps^{k_2}}  + \frac 1 {2\eps^{k_3}} |\qq|^2 \right) \mathrm{d} x.
\end{align}
Due to the singular terms with respect to $\eps$, it is not completely obvious that $\Eh$ is an approximation to $\E$.
\begin{lemma}
  Under the assumptions of Thm.~\ref{thm:aphyp}, the leading-order term of $\Eh(\ww)$ is equal to $\E(c)$, where $\ww$ is a solution to \eqref{eq:CHHExplicit} with the choice of parameters from Thm.~\ref{thm:aphyp}, and $c$ is a solution to \eqref{eq:CH}.
\end{lemma}
\begin{proof}
  From the proof of Thm.~\ref{thm:aphyp}, we know that $\cc_0 = c$ and $\pp_0 = \nabla \cc_0$, see Eq.~\eqref{eq:ppglcc}. Hence, the first two terms in $\Eh(\ww_0)$ will be equal to $\E(c)$. It remains to show that the other terms vanish as $\eps \rightarrow 0$. As $\cc_0 = \varphi_0$, this implies that $\cc-\varphi = \O(\eps)$ and hence $\frac{(\cc - \varphi)^2}{\eps} = \O(\eps)$. Furthermore, $w = \O(\eps^{k_2})$, see \eqref{eq:wkglnull}, and $\qq = \O(\eps^{k_3})$, see \eqref{eq:qqk3glnull}. This makes the last two terms of \eqref{eq:Eh} vanish asymptotically, which proves the claim.
\end{proof}

\section{Relative energy analysis}
\label{sec:relativeenergyanalysis}
In this section, we provide an a~priori error analysis framework that allows for an asymptotic comparison of strong solutions $c$ of the Cahn-Hilliard equation \eqref{eq:CH} to weak entropy solutions $\ww=(\cc,\mathbf{q},w,\mathbf{p},\varphi)$ of the relaxation system

\begin{equation}\label{eq:relaxedCHdiffusive}
	\left\{
	\begin{aligned}
		\cc_t \quad
		&+ \quad \nabla \cdot \frac{\mathbf{q}}{\kappa_3}
		&&= \quad 2\kappa_1\Delta \cc, \\
		\mathbf{q}_t \quad
		&+ \nabla\left(g'(\cc)+\frac{1}{\kappa_1}(\cc-\varphi)\right)
		&&= \quad -\frac{1}{\kappa_3}\mathbf{q}, \\
		w_t \quad
		&- \quad \nabla \cdot (\gamma \mathbf{p})
		&&= \quad \frac{1}{\kappa_1}(\cc-\varphi), \\
		\mathbf{p}_t \quad
		&- \quad \frac{1}{\kappa_2} \nabla w
		&&= \quad 0, \\
		\varphi_t \quad & \quad
		&&= \quad \frac{1}{\kappa_2}w
	\end{aligned}
	\right.,
\end{equation}
with respect to the relaxation parameters $(\kappa_1,\kappa_2,\kappa_3)\in(0,1)^3$.
To this end, we employ the relative energy framework for the (convexified) energy density
\begin{align}\label{eq:convexenergy}
    E_H:=&~g(\cc)+\frac{\gamma}{2}|\mathbf{p}|^2+\frac{1}{2\kappa_1}(\cc-\varphi)^2+\cc^2+\frac{1}{2\kappa_2}w^2+\frac{1}{2\kappa_3}|\mathbf{q}|^2\nonumber\\
    =&~g_\cup(\cc)+\frac{\gamma}{2}|\mathbf{p}|^2+\frac{1}{2\kappa_1}(\cc-\varphi)^2+\frac{\cc^2}{2}+\frac{1}{2\kappa_2}w^2+\frac{1}{2\kappa_3}|\mathbf{q}|^2,
\end{align}
where $g_\cup(\cc):=(\cc^4+1)/4$ is the convex part of $g$; see also \eqref{eq:splittingdoublewellformula}.

\begin{remark}[On the convexified energy]
    Note that the relaxation system \eqref{eq:relaxedCHdiffusive} differs from \eqref{eq:CHHExplicit} only through the addition of the diffusive term $2\kappa_1\Delta \cc$ in the evolution equation for $\cc$.
	This additional term allows us to effectively compute the time evolution of the convexified energy $E_H$, defined in \eqref{eq:convexenergy}, for weak entropy solutions $\ww$ of \eqref{eq:relaxedCHdiffusive}.
		More precisely, if $\ww$ is a weak solution of \eqref{eq:relaxedCHdiffusive} that dissipates the original energy \eqref{eq:Eh}, the regularizing effect of the Laplacian $\Delta\cc$ ensures that the time evolution of $\cc^2$ can, in fact, be directly computed via the evolution equation for $\cc$; we refer to \cite{giesselmann2026justificationrelaxationapproximationnavierstokescahnhilliard} (in particular Section 3.1) for a more detailed discussion in the context of a hyperbolic approximation of an incompressible Navier-Stokes-Cahn-Hilliard model.
\end{remark}

A comparison of $\ww$ to $c$ is obtained by interpreting $c$ as a strong solution to the relaxation system with residuals
\begin{equation}\label{eq:relaxedCHdiffusiveRes}
	\left\{
	\begin{aligned}
		\bar \cc_t \quad
		&+ \quad \nabla \cdot \frac{\bar{\mathbf{q}}}{\kappa_3}
		&&= \quad 2\kappa_1\Delta \bar\cc+ \mathcal{R}_1, \\
		\bar{\mathbf{q}}_t \quad
		&+ \nabla\left(g'(\bar\cc)+\frac{1}{\kappa_1}(\bar\cc-\bar\varphi)\right)
		&&= \quad -\frac{1}{\kappa_3}\bar{\mathbf{q}}+\mathcal{R}_2, \\
		\bar w_t \quad
		&- \quad \nabla \cdot (\gamma \bar{\mathbf{p}})
		&&= \quad \frac{1}{\kappa_1}(\bar\cc-\bar\varphi)+\mathcal{R}_3, \\
		\bar{\mathbf{p}}_t \quad
		&- \quad \frac{1}{\kappa_2} \nabla \bar w
		&&= \quad \mathcal{R}_4, \\
		\bar \varphi_t \quad & \quad
		&&= \quad \frac{1}{\kappa_2}\bar w + \mathcal{R}_5,
	\end{aligned}
	\right.
\end{equation}
for given residuals $\mathcal{R}_i$, $i\in\{1,2,3,4,5\}$, which are (at least) Lipschitz functions of space and time. Solutions of \eqref{eq:relaxedCHdiffusiveRes} are denoted by $\bar{\ww}=(\bar{\cc},\bar{\mathbf{q}},\bar{w},\bar{\mathbf{p}},\bar{\varphi})$.

\begin{remark}
    The last two evolution equations of \eqref{eq:relaxedCHdiffusive} and \eqref{eq:relaxedCHdiffusiveRes} imply that
    \begin{align}\label{eq:p=Grad(phi)}
	\mathbf{p}=\nabla\varphi
\end{align}
and
\begin{align}\label{eq:p=Grad(phi)+res}
	\bar{\mathbf{p}}=\nabla\bar{\varphi}+\mathcal{R}_0+\int_{0}^{t}\mathcal{R}_4-\nabla\mathcal{R}_5~\mathrm{d}s,
\end{align}
under the initial conditions $\mathbf{p}_0=\nabla\varphi_0$ and $\bar{\mathbf{p}}_0=\nabla\bar{\varphi}_0+\mathcal{R}_0$, respectively, where $\mathcal{R}_0$ is some given Lipschitz function.
\end{remark}

As a first step, we now study the time evolution of the relative energy
\begin{align}
\eta(\ww,\bar{\ww}):=&~g_\cup(\cc)-g_\cup(\bar{\cc})-g_\cup'(\bar{\cc})(\cc-\bar{\cc})+\frac{\gamma}{2}|\mathbf{p}-\bar{\mathbf{p}}|^2+\frac{1}{2\kappa_1}\left(\cc-\varphi-(\bar{\cc}-\bar{\varphi})\right)^2\nonumber\\
&+\frac{1}{2}(\cc-\bar{\cc})^2+\frac{1}{2\kappa_2}(w-\bar{w})^2+\frac{1}{2\kappa_3}|\mathbf{q}-\bar{\mathbf{q}}|^2
\end{align}
between a solution $\ww$ to \eqref{eq:relaxedCHdiffusive} and a solution $\bar{\ww}$ to \eqref{eq:relaxedCHdiffusiveRes}.

\begin{proposition}\label{prop:relative energy time evolution}
	Let $\ww$ be a weak entropy solution to \eqref{eq:relaxedCHdiffusive} and let $\bar{\ww}$ be a strong solution to \eqref{eq:relaxedCHdiffusiveRes}.
	Then, for every $t\leq T$, there holds
    \begin{align}
		\int_\Omega \eta\left(\ww,\bar{\ww}\right)(\xx,t)~\mathrm{d}\xx\leq&~ \int_\Omega \eta\left(\ww,\bar{\ww}\right)(\xx,0)~\mathrm{d}\xx + \int_0^t\int_\Omega~K\cdot\eta\left(\ww,\bar{\ww}\right)(\xx,s) + \mathcal{R}(\xx,s)-\mathcal{D}(\xx,s)~\mathrm{d}\xx\mathrm{d}s
	\end{align}
    with
    \begin{align}
        K:=&~18(1+\bar{\cc}^2)\left|2\kappa_1\nabla\bar{\cc}-\frac{\bar{\mathbf{q}}}{\kappa_3}\right|^2+3\left|\frac{\bar{\mathbf{q}}}{\kappa_3}\cdot\nabla\bar{\cc}\right|+32\kappa_1\left(|\nabla\bar{\cc}|^2+|\nabla\bar{\cc}^2|^2\right)+12\label{eq:K constant},\\
        \mathcal{R}:=&~12\left|\mathcal{R}_0+\int_{0}^{t}\mathcal{R}_4-\nabla\mathcal{R}_5~\mathrm{d}s\right|^2-\mathcal{R}_1\left(1+g''(\bar{\cc})\right)(\cc-\bar{\cc})-\mathcal{R}_2\cdot\left(\frac{\mathbf{q}}{\kappa_3}-\frac{\bar{\mathbf{q}}}{\kappa_3}\right)\nonumber\\
        &-\frac{1}{\kappa_2}\mathcal{R}_3(w-\bar{w})-\gamma\mathcal{R}_4(\mathbf{p}-\bar{\mathbf{p}})
        -\frac{1}{\kappa_1}(\mathcal{R}_1-\mathcal{R}_5)\left(\cc-\varphi-(\bar{\cc}-\bar{\varphi})\right),\label{eq:Residuum}\\
        \mathcal{D}:=&~\frac{1}{4}\left|\frac{\mathbf{q}}{\kappa_3}-\frac{\bar{\mathbf{q}}}{\kappa_3}\right|^2+\frac{1}{4}|\nabla(\cc-\bar{\cc})|^2+\kappa_1|\nabla(\cc-\bar{\cc})|^2+6\kappa_1(\cc-\bar{\cc})^2|\nabla\bar{\cc}|^2.
    \end{align}
\end{proposition}
\begin{proof}
    In the following, the proof is given when both $\ww$ and $\bar{\ww}$ are strong solutions.
	The generalization to weak entropy solutions $\ww$ can be obtained by the same arguments as those used in \cite{giesselmann2026justificationrelaxationapproximationnavierstokescahnhilliard}.

		We first compute the time evolution of the relative energy rather naively by using the evolution equations in \eqref{eq:relaxedCHdiffusive} and \eqref{eq:relaxedCHdiffusiveRes} and then apply the product rule to obtain
	\begin{align}\label{eq:relative energy time evolution}
		\partial_t\eta=&-\nabla\cdot\left[(g_\cup'(\cc)-g_\cup'(\bar{\cc}))\left(\frac{\mathbf{q}}{\kappa_3}-\frac{\bar{\mathbf{q}}}{\kappa_3}-2\kappa_1\nabla\left(\cc-\bar{\cc}\right)\right)+\frac{\gamma}{\kappa_2}(w-\bar{w})(\mathbf{p}-\bar{\mathbf{p}})\right]\nonumber\\
		&-\nabla\cdot\left[\frac{1}{\kappa_1}\left(\cc-\varphi-(\bar{\cc}-\bar{\varphi})\right)\left(\frac{\mathbf{q}}{\kappa_3}-\frac{\bar{\mathbf{q}}}{\kappa_3}\right)-2\left[\kappa_1(\cc-\bar{\cc})+ \left(\cc-\varphi-(\bar{\cc}-\bar{\varphi})\right)\right]\nabla\left(\cc-\bar{\cc}\right)\right]\nonumber\\
		&+\left[g_\cup'(\cc)-g_\cup'(\bar{\cc})-g_\cup''(\bar{\cc})(\cc-\bar{\cc})\right]\left(2\kappa_1\Delta \bar{\cc}-\nabla\cdot\frac{\bar{\mathbf{q}}}{\kappa_3}\right)-2\kappa_1\nabla\left(g_\cup'(\cc)-g_\cup'(\bar{\cc})\right) \cdot\nabla\left(\cc-\bar{\cc}\right)\nonumber\\
		&+2\left(\frac{\mathbf{q}}{\kappa_3}-\frac{\bar{\mathbf{q}}}{\kappa_3}\right)\cdot\nabla(\cc-\bar{\cc})-2\nabla(\cc-\bar{\cc})\cdot\nabla(\varphi-\bar{\varphi})-2(1+\kappa_1)|\nabla(\cc-\bar{\cc})|^2-\left(\frac{\mathbf{q}}{\kappa_3}-\frac{\bar{\mathbf{q}}}{\kappa_3}\right)^2\nonumber\\
		&-\mathcal{R}_1(\cc-\bar{\cc})-\gamma\mathcal{R}_4(\mathbf{p}-\bar{\mathbf{p}})-\frac{1}{\kappa_1}(\mathcal{R}_1-\mathcal{R}_5)\left(\cc-\varphi-(\bar{\cc}-\bar{\varphi})\right)-\frac{1}{\kappa_2}\mathcal{R}_3(w-\bar{w})\nonumber\\
        &-\mathcal{R}_2\cdot\left(\frac{\mathbf{q}}{\kappa_3}-\frac{\bar{\mathbf{q}}}{\kappa_3}\right)-g''(\bar{\cc})\mathcal{R}_1(\cc-\bar{\cc}).
	\end{align}
	Consequently, the terms
	\begin{align}
		I_1:=&\left[g_\cup'(\cc)-g_\cup'(\bar{\cc})-g_\cup''(\bar{\cc})(\cc-\bar{\cc})\right]\nabla\cdot\left(2\kappa_1\nabla\bar{\cc}-\frac{\bar{\mathbf{q}}}{\kappa_3}\right),\\
		I_2:=&-2\kappa_1\nabla\left(g_\cup'(\cc)-g_\cup'(\bar{\cc})\right)\cdot\nabla\left(\cc-\bar{\cc}\right),\\
		I_3:=&~2\left[\left(\frac{\mathbf{q}}{\kappa_3}-\frac{\bar{\mathbf{q}}}{\kappa_3}\right)-\nabla(\varphi-\bar{\varphi})\right]\cdot\nabla(\cc-\bar{\cc}),\label{eq:J}
	\end{align}
	need to be controlled appropriately.
	First, we note that we can rewrite
	\begin{align}\label{eq:g identities}
		g_\cup'(\cc)-g_\cup'(\bar{\cc})-g_\cup''(\bar{\cc})(\cc-\bar{\cc})=(\cc-\bar{\cc})^3+3\bar{\cc}(\cc-\bar{\cc})^2.
	\end{align}
	Thus, for a constant $k\geq 1$ to be fixed later, we have
	\begin{align}
		I_1=&~\nabla\cdot\left[\left(g_\cup'(\cc)-g_\cup'(\bar{\cc})-g_\cup''(\bar{\cc})(\cc-\bar{\cc})\right)\left(2\kappa_1\nabla\bar{\cc}-\frac{\bar{\mathbf{q}}}{\kappa_3}\right)\right]\nonumber\\
		&-3\left((\cc-\bar{\cc})^2\nabla(\cc-\bar{\cc})+(\cc-\bar{\cc})^2\nabla\bar{\cc}+2\bar{\cc}(\cc-\bar{\cc})\nabla(\cc-\bar{\cc})\right)\cdot\left(2\kappa_1\nabla\bar{\cc}-\frac{\bar{\mathbf{q}}}{\kappa_3}\right)\nonumber\\
		\leq&~\nabla\cdot\left[\left(g_\cup'(\cc)-g_\cup'(\bar{\cc})-g_\cup''(\bar{\cc})(\cc-\bar{\cc})\right)\left(2\kappa_1\nabla\bar{\cc}-\frac{\bar{\mathbf{q}}}{\kappa_3}\right)\right]\nonumber\\
		&+\frac{9k}{2}\left|\left(2\kappa_1\nabla\bar{\cc}-\frac{\bar{\mathbf{q}}}{\kappa_3}\right)\right|^2\left((\cc-\bar{\cc})^4+4\bar{\cc}^2(\cc-\bar{\cc})^2\right)+3\left|\frac{\bar{\mathbf{q}}}{\kappa_3}\cdot\nabla\bar{\cc}\right|(\cc-\bar{\cc})^2\nonumber\\
        &+\frac{1}{k}|\nabla(\cc-\bar{\cc})|^2-6\kappa_1(\cc-\bar{\cc})^2|\nabla\bar{\cc}|^2.
	\end{align}

	Concerning $I_2$, equation \eqref{eq:g identities} implies
	\begin{align*}
		g_\cup'(\cc)-g_\cup'(\bar{\cc})=(\cc-\bar{\cc})^3+3\bar{\cc}(\cc-\bar{\cc})^2+3\bar{\cc}^2(\cc-\bar{\cc})
	\end{align*}
so that
	\begin{align}
		I_2=&-6\kappa_1\left[\left((\cc-\bar{\cc})^2+\bar{\cc}^2\right)|\nabla(\cc-\bar{\cc})|^2+2\bar{\cc}(\cc-\bar{\cc})|\nabla(\cc-\bar{\cc})|^2\right]\nonumber\\
		&-6\kappa_1\left[(\cc-\bar{\cc})^2+2\bar{\cc}(\cc-\bar{\cc})\right]\nabla\bar{\cc}\cdot\nabla(\cc-\bar{\cc})\nonumber\\
		\leq&~18k\kappa_1\left(|\nabla\bar{\cc}|^2(\cc-\bar{\cc})^4+|\nabla\bar{\cc}^2|^2(\cc-\bar{\cc})^2\right)+\frac{\kappa_1}{k}|\nabla(\cc-\bar{\cc})|^2.
	\end{align}
	Finally, due to \eqref{eq:p=Grad(phi)} and \eqref{eq:p=Grad(phi)+res}, we have
	\begin{align}
		I_3\leq\frac{3}{4}\left|\frac{\mathbf{q}}{\kappa_3}-\frac{\bar{\mathbf{q}}}{\kappa_3}\right|^2+12|\mathbf{p}-\bar{\mathbf{p}}|^2+12\left|\mathcal{R}_0+\int_{0}^{t}\mathcal{R}_4-\nabla\mathcal{R}_5~\mathrm{d}s\right|^2+\frac{3}{2}|\nabla\left(\cc-\bar{\cc}\right)|^2.
	\end{align}

	Summarizing, we thus obtain
	\begin{align}
		\frac{d}{dt}\int_\Omega \eta~\mathrm{d}\xx\leq~ \int_\Omega&~ \frac{9k}{2}\left|\left(2\kappa_1\nabla\bar{\cc}-\frac{\bar{\mathbf{q}}}{\kappa_3}\right)\right|^2\left((\cc-\bar{\cc})^4+\bar{\cc}^2(\cc-\bar{\cc})^2\right)+3\left|\frac{\bar{\mathbf{q}}}{\kappa_3}\cdot\nabla\bar{\cc}\right|(\cc-\bar{\cc})^2\nonumber\\
		&+8k\kappa_1\left(|\nabla\bar{\cc}|^2(\cc-\bar{\cc})^4+|\nabla\bar{\cc}^2|^2(\cc-\bar{\cc})^2\right)+12|\mathbf{p}-\bar{\mathbf{p}}|^2\nonumber\\
		&-\frac{1}{4}\left|\frac{\mathbf{q}}{\kappa_3}-\frac{\bar{\mathbf{q}}}{\kappa_3}\right|^2-\frac{k-2}{2k}|\nabla(\cc-\bar{\cc})|^2-\kappa_1|\nabla(\cc-\bar{\cc})|^2-6\kappa_1(\cc-\bar{\cc})^2|\nabla\bar{\cc}|^2\nonumber\\
		&+12\left|\mathcal{R}_0+\int_{0}^{t}\mathcal{R}_4-\nabla\mathcal{R}_5~\mathrm{d}s\right|^2-\mathcal{R}_1(\cc-\bar{\cc})-\gamma\mathcal{R}_4(\mathbf{p}-\bar{\mathbf{p}})\nonumber\\
        &-\frac{1}{\kappa_1}(\mathcal{R}_1-\mathcal{R}_5)\left(\cc-\varphi-(\bar{\cc}-\bar{\varphi})\right)-\frac{1}{\kappa_2}\mathcal{R}_3(w-\bar{w})\nonumber\\
        &-\mathcal{R}_2\cdot\left(\frac{\mathbf{q}}{\kappa_3}-\frac{\bar{\mathbf{q}}}{\kappa_3}\right)-g''(\bar{\cc})\mathcal{R}_1(\cc-\bar{\cc})~\mathrm{d}\xx,
	\end{align}
	such that choosing $k=4$ yields the desired estimate.
\end{proof}

The explicit form of the term $\mathcal{R}$ in \eqref{eq:Residuum}, which results from the residuals, suggests constructing a solution $\bar{\ww}$ to \eqref{eq:relaxedCHdiffusiveRes} from a solution $c$ to \eqref{eq:CH} such that
\begin{equation}\label{zero-res}
    \mathcal{R}_1-\mathcal{R}_5=\mathcal{R}_3=0,
\end{equation}
in order to obtain optimal a~priori error estimates.
Indeed, using Young's inequality, we need to estimate each summand in \eqref{eq:Residuum} by the sum of a term that is controlled by $\mathcal{D}$ and the relative energy and a second term.
The rates with which these ``second'' terms go to zero for $\kappa_1, \kappa_2, \kappa_3 \rightarrow 0$ determine the rate in the final error estimate in Thm.~\ref{thm:error_estimate}.
The term in \eqref{eq:Residuum} containing $\mathcal{R}_1-\mathcal{R}_5$ also contains $\kappa_1^{-1}$, and this $\kappa_1^{-1}$ reduces the convergence rate.
Similarly, the term containing $\mathcal{R}_3$ contains $\kappa_2^{-1}$.

The property \eqref{zero-res} can be ensured as follows.

\begin{proposition}\label{prop:relaxed approximate solution from exact CH}
    Let $c\in C^2\left([0,T];C^3(\overline \Omega)\right)\cap C^3\left([0,T];C^1(\overline \Omega)\right)$ be a strong solution to \eqref{eq:CH} and set ${\mu(c):=g'(c)-\gamma\Delta c}$.
	Define $\bar{\ww}=(\bar{\cc},\bar{\mathbf{q}},\bar{w},\bar{\mathbf{p}},\bar{\varphi})$ by
    \begin{align*}
	\bar{\cc}:=&~c,\\
	\frac{\bar{\mathbf{q}}}{\kappa_3}:=&-\nabla\mu(c),\\
	\frac{\bar{w}}{\kappa_2}:=&~\Delta\mu(c)+\kappa_1\left[2\Delta c+\gamma\partial_t\Delta c\right],\\
	\bar{\mathbf{p}}:=&~\nabla c+\frac{\kappa_2}{\gamma}\partial_t\left[\nabla\mu(c)+\kappa_1\left(2\nabla c+\gamma\partial_t\nabla c\right)\right],\\
	\bar{\varphi}:=&~c+\kappa_1\gamma\Delta c.%\label{eq:hyperbolization solutions from CH last}
\end{align*}
Then $\bar{\ww}$ is a strong solution to \eqref{eq:relaxedCHdiffusiveRes} for
\begin{align*}
    \mathcal{R}_1:=&-2\kappa_1\Delta c,\\
	\mathcal{R}_2:=&-\kappa_3\partial_t\nabla\mu(c),\\
	\mathcal{R}_3:=&~0,\\
	\mathcal{R}_4:=&~\frac{\kappa_2}{\gamma}\partial_t\partial_t\left[\nabla\mu(c)+\kappa_1\left(2\nabla c+\gamma\partial_t\nabla c\right)\right]-\kappa_1\nabla\left[2\Delta c+\gamma\partial_t\Delta c\right],\\
	\mathcal{R}_5:=&-2\kappa_1\Delta c.
\end{align*}
\end{proposition}
\begin{proof}
    The statement is obtained by simply checking the identities in \eqref{eq:relaxedCHdiffusiveRes}.
\end{proof}

Finally, we can give the a~priori error estimate.

\begin{theorem}\label{thm:error_estimate}
    Let $c\in C^2\left([0,T];C^3(\overline \Omega)\right)\cap C^3\left([0,T];C^1(\overline \Omega)\right)$ be a strong solution to \eqref{eq:CH} and set \linebreak ${\mu(c):=g'(c)-\gamma\Delta c}$. Let $\ww$ be a weak entropy solution to \eqref{eq:relaxedCHdiffusive} with initial data
    \begin{align*}
        &\cc(0,\cdot)=c(0,\cdot),\quad
		\frac{\mathbf{q}}{\kappa_3}(0,\cdot)=-\nabla\mu(c)(0,\cdot),\quad
		\frac{w}{\kappa_2}(0,\cdot)=\Delta\mu(c(0,\cdot)),\\
        &\mathbf{p}(0,\cdot)=\nabla c(0,\cdot),\quad
		\varphi(0,\cdot)= c(0,\cdot)+\kappa_1\gamma\Delta c(0,\cdot).
    \end{align*}
    Then, for all $t\leq T$,
    \begin{align}\label{eq:the error estimate}
        &~\left(\|\cc-c\|+\|(\cc-c)^2\|+\sqrt{\gamma}\|\mathbf{p}-\nabla c\|+\sqrt{\kappa_1}\left\|\frac{\cc-\varphi}{\kappa_1}+\gamma\Delta c\right\|+\sqrt{\kappa_2}\left\|\frac{w}{\kappa_2}-\Delta\mu(c)\right\|+\sqrt{\kappa_3}\left\|\frac{\mathbf{q}}{\kappa_3}+\nabla\mu(c)\right\|\right)(t)\nonumber\\
        \leq&~C\exp(tK)\left(\frac{1}{K}\left(\kappa_1+\frac{\kappa_2}{\gamma}+\kappa_3\right)+\kappa_1+\frac{\kappa_2}{\sqrt{\gamma}}\right),
    \end{align}
    with $\|\cdot\|:=\|\cdot\|_{L^2(\Omega)}$.
	The constant $K>0$ is defined as the supremum norm of the quantity in \eqref{eq:K constant}, while the prefactor $C>0$ similarly depends only on the norms of $c\in C^2\left([0,T];C^3(\overline \Omega)\right)\cap C^3\left([0,T];C^1(\overline \Omega)\right)$.
	In particular, it should be noted that $C$ and $K$ implicitly depend on $\gamma$ via $c$.
\end{theorem}

\begin{remark}
	The quadratic and quartic terms $\|\cc-c\|$ and $\|(\cc-c)^2\|$ on the left-hand side of \eqref{eq:the error estimate} come from the quadratic/quartic part of $g(c)$, respectively.
\end{remark}

\begin{remark}
	The error estimate given in \eqref{eq:the error estimate} suggests using $\kappa_2$ scaled by $\gamma$, in particular because $\gamma$ is small in applications.
	This explains the choice in Thm.~\ref{thm:aphyp}.
\end{remark}

\begin{proof}
    We define a strong solution $\bar{\ww}$ to \eqref{eq:relaxedCHdiffusiveRes} with residuals based on $c$ as defined in Proposition~\ref{prop:relaxed approximate solution from exact CH}.
	Then it is clear that the left-hand side of inequality \eqref{eq:the error estimate} can be bounded by the square root of the relative energy between $\ww$ and $\bar{\ww}$ plus terms that are of order $O(\kappa_1+\kappa_2+\kappa_3)$.
	It is thus enough to prove an upper bound for the relative energy between $\ww$ and $\bar{\ww}$ that yields the scaling as stated on the right-hand side of \eqref{eq:the error estimate}.

    Since, by construction, $\mathcal{R}_1-\mathcal{R}_5=\mathcal{R}_3=0$, in the context of Proposition~\ref{prop:relative energy time evolution} we obtain
    \begin{equation}\label{eq:mathcalR}
        \mathcal{R}=~12\left|\mathcal{R}_0+\int_{0}^{t}\mathcal{R}_4-\nabla\mathcal{R}_5~\mathrm{d}s\right|^2-\mathcal{R}_1\left(1+g''(c)\right)(\cc-c)-\mathcal{R}_2\cdot\left(\frac{\mathbf{q}}{\kappa_3}-\frac{\bar{\mathbf{q}}}{\kappa_3}\right)-\gamma\mathcal{R}_4(\mathbf{p}-\bar{\mathbf{p}}).
    \end{equation}
    We now bound each term on the right-hand side of \eqref{eq:mathcalR} separately.

    First, note that
    \begin{align*}
        \mathcal{R}_0=&~\left(\frac{1}{\gamma}\left[\kappa_2\partial_t\nabla\mu(c)+\kappa_1\kappa_2\partial_t\left(2\nabla c+\gamma\partial_t\nabla c\right)\right]-\kappa_1\gamma\nabla\Delta c\right)\Bigg|_{s=0}\\
        \int_0^t\mathcal{R}_4-\nabla\mathcal{R}_5~\mathrm{d}s =&~\left[\frac{1}{\gamma}\left[\kappa_2\partial_t\nabla\mu(c)+\kappa_1\kappa_2\partial_t\left(2\nabla c+\gamma\partial_t\nabla c\right)\right]-\kappa_1\gamma \nabla\Delta c\right]^{s=t}_{s=0}.
    \end{align*}
    Thus,
    \begin{align}
        &~12\left|\mathcal{R}_0(t)+\int_{0}^{t}\mathcal{R}_4-\nabla\mathcal{R}_5~\mathrm{d}s\right|^2\nonumber\\
        \leq&~ C_0\left( \frac{\kappa_2^2}{\gamma^2}\left|\partial_t\nabla\mu(c)\right|^2+\frac{\kappa_1^2\kappa_2^2}{\gamma^2}\left(\left|\partial_t\nabla c\right|^2+\left|\gamma\partial_t\partial_t\nabla c\right|^2\right)+\kappa_1^2\gamma^2\left|\nabla\Delta c\right|^2\right)(t),
    \end{align}
    for $C_0>1$ large enough, by repeated application of $(a+b)^2\leq 2a^2+2b^2$ for $a,b\in\mathbb{R}$.

    On the other hand,
    \begin{align}
        \mathcal{R}_1\left(1+g''(c)\right)(\cc-c)\leq \frac{\kappa_1^2}{2}\left|3c^2\Delta c\right|^2+\frac{1}{2}(\cc-c)^2,
    \end{align}
    while
    \begin{align}
        \mathcal{R}_2\cdot\left(\frac{\mathbf{q}}{\kappa_3}-\frac{\bar{\mathbf{q}}}{\kappa_3}\right)\leq \kappa_3^2\left|\partial_t\nabla\mu(c)\right|^2+\frac{1}{4}\left|\frac{\mathbf{q}}{\kappa_3}-\frac{\bar{\mathbf{q}}}{\kappa_3}\right|^2,
    \end{align}
    and
    \begin{align}
        \gamma\mathcal{R}_4(\mathbf{p}-\bar{\mathbf{p}})\leq & ~k\left[\frac{\kappa_2^2}{\gamma}\left|\partial_t\partial_t\nabla\mu(c)\right|^2+\frac{\kappa_1^2\kappa_2^2}{\gamma}\left(\left|\partial_t\partial_t\nabla c\right|^2+\left|\gamma\partial_t\partial_t\partial_t\nabla c\right|^2\right)+\kappa_1^2\gamma\left(\left|\nabla\Delta c\right|^2+\left|\gamma\partial_t\nabla\Delta c\right|^2\right)\right]\nonumber\\
        &+\gamma\left|\mathbf{p}-\bar{\mathbf{p}}\right|^2,
    \end{align}
    for $k>1$ large enough, by Young's inequality.
    Thus, an application of Proposition~\ref{prop:relative energy time evolution} together with Gronwall's Lemma now yields the desired estimate.
\end{proof}

\begin{remark}
    From \eqref{eq:the error estimate}, it is evident that the relaxation provides approximations to $\nabla c, \Delta c, \nabla\mu(c), \Delta\mu(c)$ as
    \begin{align}
        \left\|\mathbf{p}-\nabla c\right\|_{L^\infty L^2}\leq&~ C_0 \left(\kappa_1+\kappa_2+\kappa_3\right),\\
        \left\|\frac{\cc-\varphi}{\kappa_1}+\gamma\Delta c\right\|_{L^\infty L^2} \leq&~ C_0 \left(\sqrt{\kappa_1}+\frac{\kappa_2}{\sqrt{\kappa_1}}+\frac{\kappa_3}{\sqrt{\kappa_1}}\right),\\
        \left\|\frac{\mathbf{q}}{\kappa_3}+\nabla\mu(c)\right\|_{L^\infty L^2} \leq&~C_0\left(\frac{\kappa_1}{\sqrt{\kappa_3}}+\frac{\kappa_2}{\sqrt{\kappa_3}}+\sqrt{\kappa_3}\right),\\
        \left\|\frac{w}{\kappa_2}-\Delta\mu(c)\right\|_{L^\infty L^2} \leq&~ C_0\left(\frac{\kappa_1}{\sqrt{\kappa_2}}+\sqrt{\kappa_2}+\frac{\kappa_3}{\sqrt{\kappa_2}}\right),
    \end{align}
    with $C_0=C_0(C,K,\gamma,T)>0$.
\end{remark}

\section{Spatial semidiscretizations}
\label{sec:spatial}
In this work, we restrict ourselves to periodic boundary conditions, and hence use periodic upwind first-derivative summation-by-parts (SBP) operators as building blocks for the spatial discretization; we follow the approaches of \cite{Mattsson2017,RanochaConservative2021}, which were similarly used in \cite{BiswasEtAl2025}.
To this end, we define $\Nx$ as the number of points when using SBP finite-difference operators, and the number of elements when using SBP discontinuous Galerkin operators, respectively.
The number of degrees of freedom is then defined as $\ndof$, which is $\ndof = \Nx$ for finite differences, and $\ndof = \Nx(p+1)$ for discontinuous Galerkin operators, where $p$ denotes the polynomial degree of the ansatz functions.
A periodic upwind SBP operator then consists of a grid $\xx \in \R^{\ndof}$, operators $D_+$ and $D_- \in \R^{\ndof\times\ndof}$ that are consistent with the first-order derivative, and a symmetric positive definite mass matrix $M$ such that
\begin{align}
 \label{eq:sbpproperty}
 MD_+ = - D_-^TM.
\end{align}
In this work, we take $M$ as a diagonal matrix. For ease of presentation, in this section, we will restrict ourselves to the one-dimensional case.
In multiple dimensions on tensor-product grids, the extension is straightforward; it is outlined in Sec.~\ref{subsec:2DRes}.

Using the upwind SBP operators $D_+$ and $D_-$ for a discretization of the Cahn-Hilliard equation \eqref{eq:CH} in such a way that discrete energy decay can be proven is possible in two ways.
Both ways use a discretization $D_2$ of the second derivative to arrive at the discrete formulation\footnote{Please note that we abuse our notation a bit, as $c$ denotes both the discrete and the continuous solution. We do not expect any confusion, and keep the notation for a cleaner formulation.}
\begin{align}
 \label{eq:CH_h}
 \partial_t c &= D_2 \left(g'(c) - \gamma D_2 c\right).
\end{align}
$D_2$ can be defined as either $D_2 = D_+ D_-$ or, in the opposite order, as $D_2 = D_- D_+$. For the sake of a unified notation, we define the quantities $D_{\circ}$ and $D_{\diamond}$, where $\circ$ denotes either plus or minus, and $\diamond$ denotes the opposite sign. Then, $D_2$ can be written as $D_2 = D_{\circ}D_{\diamond}$ for both cases.
\begin{remark}
 If the SBP operator stems from a finite-difference scheme, both definitions of $D_2$ coincide, as finite-difference matrices on periodic grids are circulant, and therefore always commute.
 For DG discretizations, this is not necessarily true, and both definitions will lead to different schemes.
\end{remark}

\begin{lemma}\label{la:masspreserving}
    The discretization \eqref{eq:CH_h} is mass-conservative, i.e., the quantity $1^T M c$ is constant in time, where $1$ denotes the vector whose entries are all equal to unity.
\end{lemma}
\begin{proof}
    From \eqref{eq:CH_h} and $\mu := g'(c) - \gamma D_2 c$, there holds
    \begin{align*}
      \partial_t (1^T M c) = 1^T M D_2 \mu = 1^T M D_{\circ}D_{\diamond} \mu
      \stackrel{\eqref{eq:sbpproperty}}= - (D_{\diamond} 1)^T M D_{\diamond} \mu = 0.
    \end{align*}
    $D_{\diamond} 1$ is zero as $D_{\diamond}$ is a consistent first-derivative operator.
\end{proof}

\begin{lemma}\label{la:energydecaych}
  A discrete version of the Cahn-Hilliard energy $\E$, see \eqref{eq:E}, is given by
  \begin{align*}
    \E_{d}(c) = 1^T M g(c) - \frac{\gamma} 2 c^T M D_2 c,
  \end{align*}
  where $1$ denotes the vector whose entries are all equal to unity.
  This energy is dissipative, i.e., for the solution $c$ to \eqref{eq:CH_h},
  \begin{align*}
    \frac{\mathrm d}{\mathrm d t} \E_{d}(c)  \leq 0.
  \end{align*}
\end{lemma}
\begin{proof}
 Before going into the proof, please note that $D_2^T M = M D_2$, which can easily be proven from the SBP properties. Then, one can compute in a straightforward way that
 \begin{align*}
  \frac{\mathrm d}{\mathrm d t} \E_{d}(c) = \left(g'(c)^T M - \gamma c^T M D_2\right) c_t
  &= \left(g'(c)^T M - \gamma c^T M D_2\right) D_2 \left(g'(c) - \gamma D_2 c\right) \\
  &= \left(M g'(c) - \gamma D_2^T M c\right)^T D_2 \left(g'(c) - \gamma D_2 c\right) \\
  &= \left(g'(c) - \gamma D_2 c\right)^T M D_2 \left(g'(c) - \gamma D_2 c\right) \leq 0.
 \end{align*}
 The last inequality is true since, for all vectors $y \in \R^{\ndof}$,
 \begin{align*}
  y^T M D_2 y = y^T M D_{\circ} D_{\diamond} y = - y^T D_{\diamond}^T M D_{\diamond} y = -(D_{\diamond}y)^T M (D_{\diamond} y) \leq 0,
 \end{align*}
 which is true since $M$ is positive definite.
\end{proof}
\begin{remark}
 $\E_d$ is, in fact, consistent with $\E$, as there holds
 \begin{align*}
  \E(c) = \int_{\Omega} \left( g(c) + \frac{\gamma}{2} |\nabla c|^2 \right) \mathrm{d}x &\approx 1^T M g(c) + \frac{\gamma}{2} (D_{\diamond} c)^T M D_{\diamond} c = 1^T M g(c) + \frac \gamma 2 c^T D_{\diamond}^T M D_{\diamond} c \\
  &= 1^T M g(c) - \frac \gamma 2 c^T M D_{\circ} D_{\diamond} c = 1^T M g(c) - \frac \gamma 2 c^T M D_2 c.
 \end{align*}
\end{remark}

Similarly to the Cahn-Hilliard equation, we define a discretization of \eqref{eq:CHHExplicit} using upwind SBP operators, which results in
\begin{subequations}\label{eq:CHHDiscExplicit}
\begin{alignat}{2}
  \label{eq:CHHDE1}  \cc_t     &+ D_{\circ} \left(\frac \qq {\eps^{k_3}} \right)              &\ = \ & 0, \\
  \label{eq:CHHDE2}  \qq_t   &+ D_{\diamond}\left(g'(\cc) + \frac{\cc-\varphi}{\eps}\right)           &\ = \ & -\frac \qq {\eps^{k_3}}, \\
  \label{eq:CHHDE3}  w_t     &- D_{\circ} (\gamma \pp)                                &\ = \ & \frac{\cc - \varphi}{\eps}, \\
  \label{eq:CHHDE4}  \pp_t   &-\frac {D_{\diamond} w} {\gamma \eps^{k_2}} &\ = \ & 0, \\
  \label{eq:CHHDE5}  \varphi_t &                                                         &\ = \ & \frac w {\gamma \eps^{k_2}}.
\end{alignat}
\end{subequations}
Again, we define this on one-dimensional domains for ease of presentation.
The $\eps$-dependence of the parameters $\kappa_1, \kappa_2, \kappa_3$ has been set explicitly according to Thm.~\ref{thm:aphyp}.
As in Thm.~\ref{thm:aphyp}, given the parameter choice in the theorem, we can show that this discretization converges to the discretization of the Cahn-Hilliard equation in \eqref{eq:CH_h} with $D_2 = D_{\circ}D_{\diamond}$. We will not show this here, but show it later in a broader context once time has been discretized.
\begin{remark}
In a similar way as before, see Lemma~\ref{la:masspreserving}, we can prove that this discretization is mass-preserving, i.e., that there holds $\partial_t(1^T M \cc) = 0$.
\end{remark}

For \eqref{eq:CHHDiscExplicit}, there also exists a discrete version of the energy $\Eh$, see Eq.~\eqref{eq:Eh}, given by
\begin{align}
  \label{eq:Ehd}
  \Ehd(\ww) = 1^T M g(\cc) + \frac \gamma 2 \|\pp\|_M^2 + \frac{\|\cc-\varphi\|_M^2}{2\eps} + \frac {\|w\|_M^2} {2 \gamma \eps^{k_2}}  + \frac 1 {2 \eps^{k_3}} \|\qq\|_M^2.
\end{align}
Here, we have defined
\begin{align}
    \label{eq:mnorm}
    \|w\|_M^2 := w^T M w.
\end{align}
\begin{lemma}\label{la:energydecaychh}
 Let the $\kappa$-parameters be chosen as in Thm.~\ref{thm:aphyp}. Then, the energy $\Ehd$ is dissipative, i.e., for a solution $\ww$ to \eqref{eq:CHHDiscExplicit}, there holds
 \begin{align*}
  \frac{\mathrm d}{\mathrm d t} \Ehd(\ww)  \leq 0.
 \end{align*}
\end{lemma}
\begin{proof}
 We can compute in a straightforward way that
 \begin{align*}
    \frac{\mathrm d}{\mathrm d t} \Ehd(\ww)
    &= g'(\cc)^T M \cc_t + \frac 1 \eps (\cc - \varphi)^T M \cc_t + \frac 1 {\eps^{k_3}} \qq^T M \qq_t + \frac 1 {\gamma \eps^{k_2}} w^T M w_t + \gamma \pp^T M \pp_t  + \frac 1 \eps (\varphi-\cc)^T M \varphi_t \\
    &= -g'(\cc)^T M D_{\circ}\left(\frac{\qq}{\eps^{k_3}}\right) - \frac 1 \eps (\cc - \varphi)^T M D_{\circ}\left(\frac{\qq}{\eps^{k_3}}\right) -
        \frac 1 {\eps^{k_3}} \qq^T M \left(\frac \qq {\eps^{k_3}} + D_{\diamond}\left(g'(\cc) + \frac{\cc-\varphi}{\eps}\right) \right) \\
        &\hphantom{= \ \ } + \frac 1 {\gamma \eps^{k_2}} w^T M \left(\frac{\cc - \varphi}{\eps} + D_{\circ} (\gamma \pp) \right) +
        \gamma \pp^T M \left(\frac {D_{\diamond} w} {\gamma \eps^{k_2}} \right)+ \frac 1 \eps (\varphi-\cc)^T M \left(\frac w {\gamma \eps^{k_2}} \right) \\
    &= -g'(\cc)^T M D_{\circ}\left(\frac{\qq}{\eps^{k_3}}\right) - \frac 1 \eps (\cc - \varphi)^T M D_{\circ}\left(\frac{\qq}{\eps^{k_3}}\right) -
        \left( \frac {\qq^T} {\eps^{k_3}}\right) M D_{\diamond}\left(g'(\cc) + \frac{\cc-\varphi}{\eps}\right) \\
        &\hphantom{= \ \ } + \frac 1 {\eps^{k_2}} w^T M \left(D_{\circ} \pp \right) +
        \pp^T M \left(\frac {D_{\diamond} w} {\eps^{k_2}} \right) - \frac 1 {\eps^{2k_3}} \|\qq\|_M^2 \\
    &= -g'(\cc)^T M D_{\circ}\left(\frac{\qq}{\eps^{k_3}}\right) - \frac 1 \eps (\cc - \varphi)^T M D_{\circ}\left(\frac{\qq}{\eps^{k_3}}\right) +
    \left( M D_{\circ} \frac {\qq} {\eps^{k_3}}\right)^T \left(g'(\cc) + \frac{\cc-\varphi}{\eps}\right) \\
        &\hphantom{= \ \ } + \frac 1 {\eps^{k_2}} w^T M \left(D_{\circ} \pp \right) -
        \left(D_{\circ}  \pp\right)^T M \left(\frac {w} {\eps^{k_2}} \right) - \frac 1 {\eps^{2k_3}} \|\qq\|_M^2 \\
    &= - \frac 1 {\eps^{2k_3}} \|\qq\|_M^2   \leq 0,
 \end{align*}
 which is obviously the desired result.
\end{proof}
\begin{remark}
	Note that from \eqref{eq:qk3}, we obtain that for the continuous solution to \eqref{eq:CHHExplicit}, there holds
	\begin{align*}
		- \frac 1 {\eps^{2k_3}} \qq^T \qq =
		- \left(\nabla \left(g'(\cc_0) -\gamma \Delta \cc_0\right)\right)^T
		\left(\nabla \left(g'(\cc_0) -\gamma \Delta \cc_0\right)\right) + \O(\eps),
	\end{align*}
	which is in line with the estimate for the energy decay of the discrete Cahn-Hilliard equation, see Lemma~\ref{la:energydecaych}.
	In particular, the energy estimate does not blow up with $\eps \rightarrow 0$.
\end{remark}

\section{Time discretization}
\label{sec:time}
In this section, we discretize \eqref{eq:CHHDiscExplicit} in time using a globally-stiffly-accurate ARS-type IMEX Runge-Kutta \cite{Ascher1997} scheme.
Since the seminal work of Eyre \cite{1998Eyre}, see also \cite{christlieb2013unconditionally}, it is known that, with respect to energy stability, it is beneficial to split the double-well potential $g$ into a convex and a concave part, and treat the latter explicitly and the former implicitly.
If the remaining terms of the equation are treated implicitly, it can be shown that for an ARS-111 discretization, this leads to a fully discrete energy-stable scheme, see \cite{1998Eyre}.
Also, for higher-order Runge-Kutta schemes, our experience shows that, e.g., the convergence of algebraic solvers is improved when using this splitting.
Recently, however, there has been some criticism of this splitting, at least in the context of the Allen-Cahn equation, see \cite{2026Dondl}. Nevertheless, the splitting is frequently used, so for the sake of an easier comparison, we also build our time integration based on Eyre's splitting.

In order to fix notation, define
\begin{align}
	\label{eq:splittingdoublewell}
	g(c) = g_{\cup}(c) + g_{\cap}(c),
\end{align}
where $g_{\cup}$ denotes the convex part and $g_{\cap}$ denotes the concave part. For the double-well potential \eqref{eq:doublewell}, a straightforward decomposition is
\begin{align}
    \label{eq:splittingdoublewellformula}
	g_{\cup}(c) = \frac {c^4+1} 4, \quad g_{\cap}(c) = -\frac {c^2} 2.
\end{align}

A type~II IMEX Runge-Kutta scheme for an additive ODE $u'(t) = F_E(u) + F_I(u)$ is given by
\begin{align*}
   \uu = u^n + \dt \left(\widetilde \alpha F_E(u^n) + \widehat {\widetilde A} F_E(\uu) + \alpha F_I(u^n) + \widehat A F_I(\uu)  \right),
\end{align*}
where $\uu$ denotes the collection of all the stages, beginning from the second stage, i.e.,
\begin{align}
	\label{eq:collectionstages}
	\uu = \begin{pmatrix}
		u^{(2)} \\ \vdots \\ u^{(s)}
	\end{pmatrix}.
\end{align}
The first stage is, by construction, given by $u^{(1)} = u^n$.
For the Butcher tableau of a type~II scheme, see Tbl.~\ref{tbl:typeii}.
Type~II schemes assume that the matrix $\widehat A$ is invertible. The update step is then given by
\begin{align*}
	u^{n+1} = u^n + \dt \left(\widetilde \beta F_E(u^n) + \widehat{\widetilde b^T} F_E(\uu) + \beta F_I(u^n) + \widehat b^T F_I(\uu) \right).
\end{align*}
Our theorems will rely on a globally-stiffly-accurate (GSA) structure, i.e., it is assumed that the Butcher tableau of the Runge-Kutta scheme is such that
\begin{align*}
	u^{n+1} = u^{(s)}.
\end{align*}

If we apply such an IMEX Runge-Kutta scheme to \eqref{eq:CH_h}, we obtain a fully discrete approximation scheme for the Cahn-Hilliard equation \eqref{eq:CH}, given by
\begin{align}
    \label{eq:CH_ht}
    \vec{c} = c^n + \dt D_2 \left(
    \tilde \alpha g'_{\cap} (c^n) +
    \widehat {\widetilde A} g'_{\cap}(\vec{c}) +
    \alpha \left(g'_{\cup} (c^n) - \gamma D_2 c^n\right) +
    \widehat A \left(g'_{\cup} (\vec c) - \gamma D_2 \vec c\right)
    \right).
\end{align}

\begin{remark}
	Type~I schemes can be written as type~II schemes with $\alpha = 0 = \tilde \alpha$ and $\tilde \beta = 0 = \beta$, so all the statements in this section apply to type~I schemes as well.
\end{remark}

\begin{table}
\begin{equation}
	\renewcommand\arraystretch{1.3}
	\begin{array}{c|ccc}
		0 & 0 \\
		\widehat{\widetilde{c}} & \widetilde{\alpha} & \widehat{\widetilde{A}} \\
		\hline
		& \widetilde{\beta} & \widehat{\widetilde{b}}^T
	\end{array}
	\qquad
	\begin{array}{c|ccc}
		0 & 0 \\
		\mathfrak{c} & \mathfrak{\alpha} & \widehat{A} \\
		\hline
		& \beta & \widehat{b}^T
	\end{array}
    \end{equation}
    \begin{equation}
    \begin{array}{c|ccc}
       0 & 0 & 0 & 0 \\
       \gamma & \gamma & 0 & 0 \\
       1 & \delta & 1-\delta & 0 \\ \hline
       & \delta & 1-\delta & 0
    \end{array} \quad
    \begin{array}{c|ccc}
       0 & 0 & 0 & 0\\
       \gamma & 0 & \gamma & 0 \\
       1 & 0 & 1-\gamma & \gamma \\ \hline
       & 0 & 1 -\gamma & \gamma
    \end{array}
    \hspace{2cm}
    \begin{array}{c|ccccc}
       0 & 0 & 0 & 0 & 0 \\[0.2em]
       \frac 1 2& \frac 1 2 & 0 & 0 & 0 \\[0.2em]
       \frac 2 3& \frac{11}{18} & \frac 1 {18} & 0 & 0 & 0 \\[0.2em]
       \frac 1 2& \frac 5 6 & - \frac 5 6 & \frac 1 2 & 0 & 0 \\[0.2em]
       1 & \frac 1 4 & \frac 7 4 & \frac 3 4 & -\frac 7 4 & 0 \\[0.2em] \hline
       & \frac 1 4 & \frac 7 4 & \frac 3 4 & -\frac 7 4 & 0
    \end{array} \quad
    \begin{array}{c|ccccc}
       0 & 0 & 0 & 0 & 0 & 0\\[0.2em]
       \frac 1 2 & 0 & \frac 1 2 & 0 & 0 & 0 \\[0.2em]
       \frac 2 3& 0 & \frac 1 6 & \frac 1 2 & 0 & 0 \\[0.2em]
       \frac 1 2 & 0 & -\frac 1 2 & \frac 1 2 & \frac 1 2 & 0 \\[0.2em]
       1 & 0 & \frac 3 2 & -\frac 3 2 & \frac 1 2 & \frac 1 2 \\[0.2em] \hline
       & 0 & \frac 3 2 & -\frac 3 2 & \frac 1 2 & \frac 1 2
    \end{array}
\end{equation}
\caption{The Butcher tableau of a type~II IMEX Runge-Kutta scheme (top). If the scheme is GSA, then the update coefficients $\widetilde \beta, \widehat{\widetilde{b}}$ and $\beta, \widehat{b}$ coincide with the row of the tableau corresponding to the last stage.
In this work, we use ARS-222 and ARS-443 schemes, with tableaux given in the bottom row on the left and right, respectively. We define $\gamma = \frac{2-\sqrt{2}} 2$ and $\delta = 1 - \frac 1 {2\gamma}$ as in \cite{Ascher1997}.
}\label{tbl:typeii}
\end{table}

Applying an IMEX Runge-Kutta scheme to \eqref{eq:CHHDiscExplicit} and using the splitting of the double-well potential \eqref{eq:splittingdoublewell} results in the fully discrete scheme
\begin{subequations}\label{eq:CHHDIMEX}
\begin{align}
	\label{eq:CHHDI1}
	\vec{\cc} &= \cc^n - \Delta t \alpha D_{\circ} \left(\frac{\qq^{n}}{\eps^{k_3}}\right) - \Delta t \widehat{A} D_{\circ} \left(\frac{\vec{\qq}}{\eps^{k_3}}\right), \\
	\label{eq:CHHDI2}
	\vec{\qq} &= \qq^n - \Delta t \widetilde{\alpha} D_{\diamond} g_{\cap}'(\cc^{n}) - \Delta t \widehat{\widetilde{A}} D_{\diamond} g_{\cap}'(\vec{\cc})
	- \Delta t \alpha D_{\diamond} g_{\cup}'(\cc^{n}) - \Delta t \widehat{A} D_{\diamond} g_{\cup}'(\vec {\cc})
	\\
	\notag
	&\qquad\;
	- \Delta t \alpha D_{\diamond} \left( \frac{\cc^{n} - \varphi^{n}}{\eps}\right)
	- \Delta t \widehat{A} D_{\diamond} \left(\frac{\vec{\cc} - \vec{\varphi}}{\eps}\right)
	- \Delta t \alpha \left(\frac{\qq^{n}}{\eps^{k_3}}\right) - \Delta t \widehat{A} \left(\frac {\vec{\qq}} {\eps^{k_3}}\right), \\
	\label{eq:CHHDI3}
	\vec{w} &= w^n + \gamma \Delta t \alpha D_{\circ} \pp^{n} + \gamma \Delta t \widehat{A} D_{\circ} \vec{\pp}
	+ \Delta t \alpha \left(\frac{\cc^{n}-\varphi^{n}}{\eps}\right) + \Delta t \widehat{A} \left(\frac{\vec{\cc}-\vec{\varphi} }{\eps}\right), \\
	\label{eq:CHHDI4}
	\vec{\pp} &= \pp^n + \Delta t \alpha D_{\diamond} \left(\frac{w^{n}}{\gamma \eps^{k_2}}\right) + \Delta t \widehat{A} D_{\diamond} \left(\frac{\vec {w}}{\gamma \eps^{k_2}}\right), \\
	\label{eq:CHHDI5}
	\vec{\varphi} &= \varphi^n + \Delta t \alpha \left(\frac{w^{n}}{\gamma \eps^{k_2}}\right) + \Delta t \widehat{A} \left(\frac{\vec{w}}{\gamma \eps^{k_2}}\right).
\end{align}
\end{subequations}
Vector arrows over a quantity denote the collection of all the stages, starting from the second one, as in \eqref{eq:collectionstages}.

\begin{theorem}\label{thm:apimex}
    Let $k_2, k_3 \in \N$ in \eqref{eq:CHHDIMEX} denote positive integers, cf. Thm.~\ref{thm:aphyp}.
    We consider a globally stiffly accurate IMEX Runge-Kutta scheme of type~II with the Butcher tableau as in Tbl.~\ref{tbl:typeii} and invertible matrix $\widehat A$.
	We assume that all the unknown values in \eqref{eq:CHHDIMEX} can be written in terms of a Hilbert expansion, which means, for example, that
	\begin{align*}
		\vec{\cc} = \vec{\cc}_0 + \eps \vec{\cc}_1 + \eps^2 \vec{\cc}_2 + \cdots.
	\end{align*}
	We assume that the initial conditions are well-prepared in the sense that
	\begin{align*}
		\qq^0 = \O(\eps^{k_3}),
		\quad w^0 = \O(\eps^{k_2}),
		\quad \pp^0 = D_{\diamond} \varphi^0 + \O(\eps).
	\end{align*}
	Furthermore, if $\alpha \neq 0$\footnote{If $\alpha = 0$, this means we consider an ARS-type method \cite{Ascher1997}.}, we assume that
	\begin{align*}
		\varphi^{0} = \cc^{0} + \O(\eps).
	\end{align*}
	Then, for $\eps \rightarrow 0$, \eqref{eq:CHHDIMEX} reduces to a discretization of the Cahn-Hilliard equation in the sense that $\cc_0^n$ is the result of an IMEX Runge-Kutta discretization of \eqref{eq:CH_ht}.
\end{theorem}
\begin{proof}
	The proof is a formal induction on $n$, where it is assumed that the well-preparedness conditions hold at time-level $n$. We will show that the well-preparedness conditions also hold for time-level $n+1$ by showing that they hold for all the stages.
	Then, using the GSA property, which means that the last stage is equal to the update, well-preparedness for time-level $n+1$ follows trivially.

	A formal analysis analogous to the proof of Thm.~\ref{thm:aphyp} is applied, i.e., the Hilbert expansion is inserted into \eqref{eq:CHHDIMEX}. Eq.~\eqref{eq:CHHDI1} immediately yields the identities
	\begin{align}
			\alpha D_{\circ} \qq^{n}_k + \widehat{A} D_{\circ} \vec{\qq}_k = 0, &\qquad 0 \leq k \leq k_3-1, \\
			\label{eq:intmed5}
			\vec{\cc}_0 = \cc^n_0 - \Delta t \alpha D_{\circ} \qq^{n}_{k_3} - \Delta t \widehat{A} D_{\circ} \vec{\qq}_{k_3}. &\quad
	\end{align}
	From Eq.~\eqref{eq:CHHDI2}, we obtain
	\begin{align}
		\label{eq:intmed1}
		\alpha \qq^n_{k} + \widehat A \vec{\qq}_{k} = 0, \quad 0 \leq k \leq k_3-2, \\
		\label{eq:intmed2}
		\alpha D_{\diamond} (\cc^{n}_0 - \varphi^{n}_0) + \widehat{A} D_{\diamond} (\vec{\cc}_0 - \vec{\varphi}_0) = -\alpha \qq^{n}_{k_3-1} - \widehat{A} \vec{\qq}_{k_3-1}
	\end{align}
	and
	\begin{align}
		\label{eq:intmed3}
		\vec{\qq}_0
		&=
		\qq^{n}_0
		- \Delta t \widetilde{\alpha} D_{\diamond} g_{\cap}'(\cc^{n}_0) - \Delta t \widehat{\widetilde{A}} D_{\diamond} g_{\cap}'(\vec{\cc}_0)
		- \Delta t \alpha D_{\diamond} g_{\cup}'(\cc^{n}_0) - \Delta t \widehat{A} D_{\diamond} g_{\cup}'(\vec{\cc}_0)
		\\
		&\qquad\;
		- \Delta t \alpha D_{\diamond} (\cc^{n}_1 - \varphi^{n}_1) - \Delta t \widehat{A} D_{\diamond} (\vec{\cc}_1 - \vec{\varphi}_1)
		- \Delta t \alpha \vec{\qq}^{n}_{k_3} - \Delta t \widehat{A} \vec{\qq}_{k_3}.
	\end{align}
	Eq.~\eqref{eq:CHHDI3} gives us
	\begin{align}
		\label{eq:intmed8}
		\vec{\cc}_0 = \vec{\varphi}_0.
	\end{align}
	Please note that if $\alpha = 0$, we do not need to impose $\cc^n_0 = \varphi^n_0$ to obtain this result; otherwise, we do.
	Together with \eqref{eq:intmed1} and \eqref{eq:intmed2}, we arrive at
	\begin{align*}
		\vec{\qq}_k = 0, \qquad 0 \leq k \leq k_3-1
	\end{align*}
	as $\qq_k^n = 0$ for $0 \leq k \leq k_3-1$ by the assumption of well-prepared initial conditions. Plugging this into \eqref{eq:intmed3}, we obtain
	\begin{align}
	\label{eq:intmed4}
	 \widetilde{\alpha} D_{\diamond} g_{\cap}'(\cc^{n}_0) + \widehat{\widetilde{A}} D_{\diamond} g_{\cap}'(\vec{\cc}_0)
	+ \alpha D_{\diamond} g_{\cup}'(\cc^{n}_0) + \widehat{A} D_{\diamond} g_{\cup}'(\vec{\cc}_0)
	+ \alpha D_{\diamond} (\cc^{n}_1 - \varphi^{n}_1) + \widehat{A} D_{\diamond} (\vec{\cc}_1 - \vec{\varphi}_1) &=
	- \alpha \vec{\qq}^{n}_{k_3} - \widehat{A} \vec{\qq}_{k_3}.
\end{align}
Continuing with \eqref{eq:CHHDI3}, we arrive at
\begin{align*}
	\vec{w}_0 &= w_0^n + \gamma \Delta t \alpha D_{\circ} \pp_0^{n} + \gamma \Delta t \widehat{A} D_{\circ} \vec{\pp}_0
	+ \Delta t \alpha \left(\cc^{n}_1-\varphi^{n}_1\right) + \Delta t \widehat{A} \left({\vec{\cc}_1-\vec{\varphi}_1 }\right).
\end{align*}
Together with
\begin{align*}
	\vec{w}_k = 0, \qquad 0 \leq k \leq k_2-1
\end{align*}
obtained from \eqref{eq:CHHDI5}, this results in
\begin{align}
	\label{eq:intmed6}
	0 &= \gamma  \alpha D_{\circ} \pp_0^{n} + \gamma  \widehat{A} D_{\circ} \vec{\pp}_0
+  \alpha \left(\cc^{n}_1-\varphi^{n}_1\right) +  \widehat{A} \left({\vec{\cc}_1-\vec{\varphi}_1 }\right).
\end{align}
We obtain from \eqref{eq:CHHDI4}--\eqref{eq:CHHDI5} that
\begin{align*}
	\vec{\pp}_0 &= \pp^n_0 + \frac{\Delta t}{\gamma} \alpha D_{\diamond} {w^{n}_{k_2}} + \frac{\Delta t}{\gamma} \widehat{A} D_{\diamond} {\vec {w}_{k_2}}, \\
	\vec{\varphi}_0 &= \varphi^n_0 + \frac{\Delta t}{\gamma} \alpha {w^{n}_{k_2}} + \frac{\Delta t}{\gamma} \widehat{A} {\vec{w}_{k_2}};
\end{align*}
which can be combined into
\begin{align*}
	\vec{\pp}_0 &= \pp^n_0 + D_{\diamond}(\vec{\varphi}_0 - \varphi_0^n).
\end{align*}
Together with the well-preparedness condition $\pp_0^n = D_{\diamond}\varphi_0^n$, this last equation leads to
\begin{align}
	\label{eq:intmed7}
	\vec{\pp}_0 &= D_{\diamond}\vec{\varphi}_0.
\end{align}
What remains is a straightforward calculation starting from \eqref{eq:intmed5}:
\begin{align*}
	\vec{\cc}_0
	&= \cc^n_0 - \Delta t \alpha D_{\circ} \qq^{n}_{k_3} - \Delta t \widehat{A} D_{\circ} \vec{\qq}_{k_3} \\
	&\stackrel{\eqref{eq:intmed4}}{=} \cc^n_0 + \Delta t D_{\circ} D_{\diamond} \left(
	\widetilde{\alpha} g_{\cap}'(\cc^{n}_0) + \widehat{\widetilde{A}} g_{\cap}'(\vec{\cc}_0)
	+ \alpha g_{\cup}'(\cc^{n}_0) + \widehat{A}  g_{\cup}'(\vec{\cc}_0)
	+ \alpha (\cc^{n}_1 - \varphi^{n}_1) + \widehat{A}  (\vec{\cc}_1 - \vec{\varphi}_1)
	\right) \\
	&\stackrel{\eqref{eq:intmed6}}{=} \cc^n_0 + \Delta t D_{\circ} D_{\diamond} \left(
	\widetilde{\alpha} g_{\cap}'(\cc^{n}_0) + \widehat{\widetilde{A}} g_{\cap}'(\vec{\cc}_0)
	+ \alpha g_{\cup}'(\cc^{n}_0) + \widehat{A}  g_{\cup}'(\vec{\cc}_0)
	-
	\gamma  \alpha D_{\circ} \pp_0^{n} - \gamma  \widehat{A} D_{\circ} \vec{\pp}_0
	\right)  \\
	&=  \cc^n_0 + \Delta t D_{\circ} D_{\diamond} \left(
	\widetilde{\alpha} g_{\cap}'(\cc^{n}_0) + \widehat{\widetilde{A}} g_{\cap}'(\vec{\cc}_0)
	+ \alpha g_{\cup}'(\cc^{n}_0) + \widehat{A}  g_{\cup}'(\vec{\cc}_0)\right)
	- \gamma \dt D_{\circ}D_{\diamond} D_{\circ} \left(\alpha \pp_0^n + \widehat A \vec{\pp}_0 \right)\\
	&\stackrel{\eqref{eq:intmed7}}=  \cc^n_0 + \Delta t D_{\circ} D_{\diamond} \left(
	\widetilde{\alpha} g_{\cap}'(\cc^{n}_0) + \widehat{\widetilde{A}} g_{\cap}'(\vec{\cc}_0)
	+ \alpha g_{\cup}'(\cc^{n}_0) + \widehat{A}  g_{\cup}'(\vec{\cc}_0)\right)
	- \gamma \dt D_{\circ}D_{\diamond} D_{\circ} D_{\diamond} \left(\alpha \varphi_0^n + \widehat A \vec{\varphi}_0 \right)\\
	&\stackrel{\eqref{eq:intmed8}}=  \cc^n_0 + \Delta t D_{\circ} D_{\diamond} \left(
	\widetilde{\alpha} g_{\cap}'(\cc^{n}_0) + \widehat{\widetilde{A}} g_{\cap}'(\vec{\cc}_0)
	+ \alpha g_{\cup}'(\cc^{n}_0) + \widehat{A}  g_{\cup}'(\vec{\cc}_0)\right)
	- \gamma \dt D_{\circ}D_{\diamond} D_{\circ} D_{\diamond} \left(\alpha \cc_0^n + \widehat A \vec{\cc}_0 \right)\\
	&=  \cc^n_0 + \Delta t D_2 \left(
	\widetilde{\alpha} g_{\cap}'(\cc^{n}_0) + \widehat{\widetilde{A}} g_{\cap}'(\vec{\cc}_0)
	+ \alpha g_{\cup}'(\cc^{n}_0) + \widehat{A}  g_{\cup}'(\vec{\cc}_0)
	- \gamma D_2 \left(\alpha \cc_0^n + \widehat A \vec{\cc}_0 \right)\right),
\end{align*}
where $D_2$ has been defined as $D_2 = D_{\circ} D_{\diamond}$.
This is the IMEX Runge-Kutta scheme applied to \eqref{eq:CH_h}---compare with \eqref{eq:CH_ht}---which concludes the proof.
\end{proof}

\section{Numerical results}
\label{sec:numres}
In this section, we show numerical results of our proposed method both in one and in two dimensions. We show results for two different types of SBP discretizations:
\begin{itemize}
	\item Upwind SBP finite-difference operators of various orders as explained in the appendix of \cite{Mattsson2017}. These schemes will simply be denoted by SBP schemes.
	\item Nodal local DG schemes of various orders placed in the framework of upwind SBP schemes as explained in \cite{RanochaConservative2021}. These schemes will be denoted by DG schemes.
\end{itemize}
The numerical results have been obtained in MATLAB; the corresponding code can be found in the reproducibility repository \cite{CHH2025Reproducibility}.
We obtained the coefficients of the SBP operators from the Julia package SummationByPartsOperators.jl \cite{ranocha2021sbp}.
We used \texttt{ode15s} \cite{shampine1997matlab} to compute some reference solutions.
For all results, we compute the error at the final time $\Te$ in the $M$-norm defined in \eqref{eq:mnorm}, which is a discrete version of the $L^2$-norm.
Please note that the $M$-norm depends on the SBP scheme under consideration through the mass matrix $M$.

\subsection{One-dimensional test case}\label{sec:onedsbp}

In this section, we consider the one-dimensional Cahn-Hilliard equation \eqref{eq:CH} with initial conditions
\begin{align}
	\label{eq:chh_initial_cos}
	c(x, 0) = \cos\left(\frac{\pi x} 5 \right)
\end{align}
in the domain $\Omega = [-5, 5]$.
While the initial conditions are rather smooth, for $\gamma \rightarrow 0$, the solution will converge for large $t$ to a discontinuous step function.
It is hence obvious, and we can see this in many of our experiments, that the spatial error typically dominates the temporal error by far.
In all the numerical experiments in this section, we compute the error at the final time $\Te = 5$.
Depending on the context, we compute three different error quantities:
\begin{itemize}
    \item When considering convergence of the discrete Cahn-Hilliard equation \eqref{eq:CH_ht} to the continuous Cahn-Hilliard equation \eqref{eq:CH}, we define
    \begin{align*}
        e_{CH} := \|c(\xx, \Te) - c^N \|_M,
    \end{align*}
    where $N$ has been chosen such that $t^N = \Te$, and $c(\xx, \Te)$ denotes the solution to \eqref{eq:CH} evaluated at the final time $\Te$ and at discrete points $\xx$ given by the SBP scheme.
    \item When considering convergence of the discrete hyperbolic Cahn-Hilliard equation \eqref{eq:CHHDIMEX} towards the discrete Cahn-Hilliard equation \eqref{eq:CH_ht}, we define the two error quantities
    \begin{align*}
        e_{hyp} &:= \|\cc^N - c^N\|_M, \\
        e_{\nabla, hyp} &:= \|\pp^N - D_{\diamond}c^N\|_M.
    \end{align*}
    In both cases, solutions to \eqref{eq:CHHDIMEX} and \eqref{eq:CH_ht} have been computed with the same parameters concerning the choice of the SBP scheme, the order of accuracy, and the time integration scheme.
\end{itemize}

As we do not have an analytical solution to \eqref{eq:CH} at our disposal, we use a highly refined numerical solution as a reference solution.
More precisely, for the reference solution, we use an eighth-order SBP finite-difference scheme on $N_x = 5120$ grid points.
Time integration for obtaining this reference solution is performed using the built-in MATLAB solver \texttt{ode15s} (for small values of $\gamma$, the system becomes stiff) with a relative tolerance of $10^{-6}$.

The resulting nonlinear systems of algebraic equations are solved using a damped Newton method with a maximum of 20 Newton steps, and a relative and absolute tolerance of $10^{-12}$.
Please note that if the tolerance is not met, the algorithm will stop after 20 Newton steps and continue from there.
We chose the parameters in a way to make sure that our results are not significantly influenced by the quality of Newton's method.
The linear systems of equations are solved using MATLAB's backslash operator on the equilibrated matrix.

\paragraph{Convergence as $\eps \rightarrow 0$}
In this paragraph, we numerically analyze the influence of $\eps$ as in Thm.~\ref{thm:aphyp}.
It is obvious that with $\eps \rightarrow 0$, some eigenvalues of the hyperbolic system \eqref{eq:CHHExplicit} diverge to infinity, resulting in a very stiff system to be solved.
Hence, in practice, the eigenvalues of the hyperbolic system should be as small as possible.
This means that $\eps$ should be as large as possible, while still being small enough such that the hyperbolic system is still a decent approximation to the original equation.
It is exactly this investigation that we are pursuing in this paragraph.

\begin{figure}
    \centering
	\includegraphics[width=0.4\textwidth]{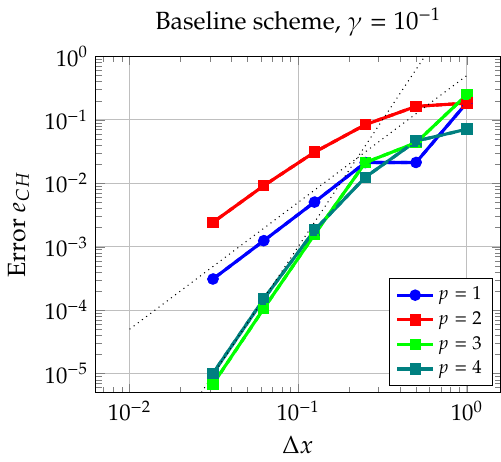}
	\includegraphics[width=0.4\textwidth]{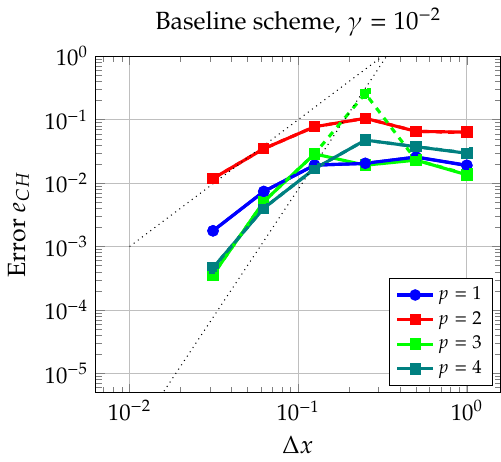}
	\caption{Convergence analysis of the original SBP scheme for ARS-222 (solid lines) and ARS-443 (dashed lines) time integration.
		The difference between ARS-222 and ARS-443 is hardly visible (only for $\gamma = 10^{-2}$ and $p=3$, there is a small outlier), as the solution is no longer in the transient regime, meaning that the solution is nearly at steady state.
		The initial condition is given by \eqref{eq:chh_initial_cos}, and we set $\Te = 5$. The number of time intervals $N_T$ is chosen to be equal to $N_x$, the spatial resolution.
        Dotted black lines indicate second and fourth order of convergence, respectively. }\label{fig:dg_err_exact}
\end{figure}

To analyze the influence of the parameters, we first report the error of the original SBP solution \eqref{eq:CH_h} with respect to the reference solution of \eqref{eq:CH} to obtain a ground truth.
Convergence results are reported in Fig.~\ref{fig:dg_err_exact} for different values of $\dx := \frac {10} {N_x}$ (the spatial resolution), $p$ (the order of accuracy of the SBP finite-difference operator), and $\gamma$.
The number of time steps $N_T$ is chosen to be equal to $N_x$. As a time integrator, the ARS-222 (solid) and the ARS-443 (dashed) schemes are chosen.
It can be seen from the figures that both schemes produce nearly the same results and that the order of convergence is the order of the spatial discretization.
It is hence safe to say that the numerical solutions' spatial error dominates the temporal error. The order of convergence is two for $p = 1$ and $p=2$, respectively, and four for $p=3$ and $p=4$.
This is expected because, by construction, the second-order finite difference operator $D_2$ will converge with an even order.
Furthermore, it can be seen that the $p=1$ case behaves better than the $p=2$ case.
Also, this is not unexpected, as, in the authors' experience, for finite differences the scheme with the same order, but smaller stencil, is typically favorable. This distinction is also visible for $p=3$ and $p=4$, but much less pronounced.

\begin{figure}
	\begin{tabular}{cc}
        \includegraphics[scale=0.7]{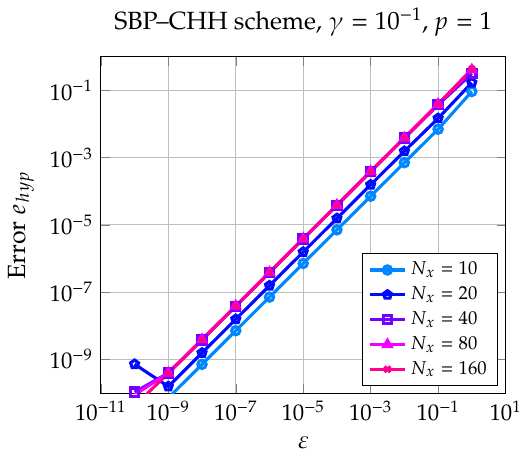} &
		\includegraphics[scale=0.7]{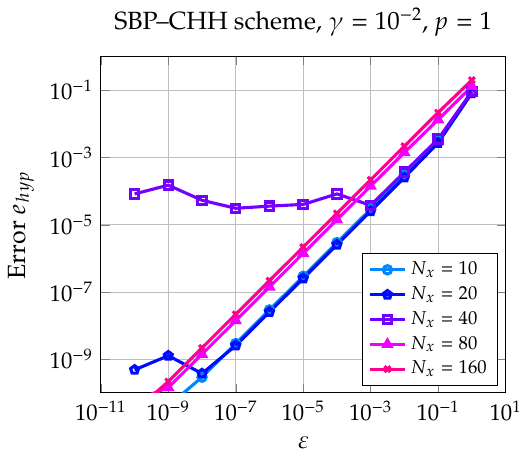} \\
        \includegraphics[scale=0.7]{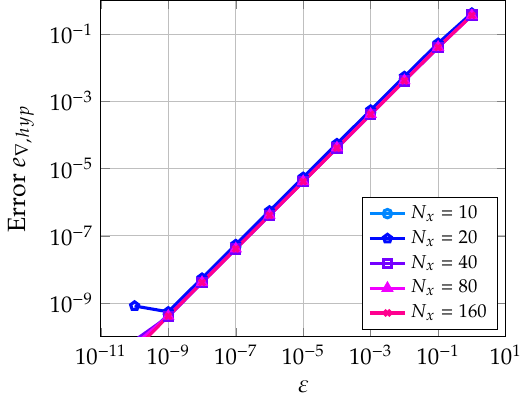} &
		\includegraphics[scale=0.7]{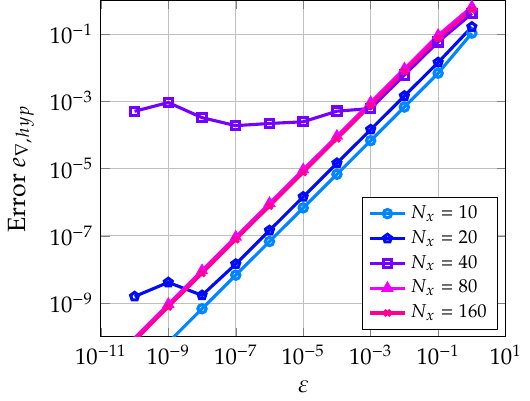} \\

        \small
		\begin{tabular}{|c|c|c|c|c|c|}
			\hline
			$N_x$ & 10 & 20 & 40 & 80 & 160 \\ \hline
			Error & 1.9e-1   &  2.1e-2  &   2.1e-2  &   5.1e-3  &   1.2e-3 \\ \hline
			EoC(E) & -- & 3.1 &  0.007  & 2.1  & 2.0   \\ \hline
			$\eps_0$ & 2.0 & 1.3e-1 & 5.7e-2 & 1.3e-2 & 3.2e-3\\ \hline
			EoC($\eps_0$) & -- & 3.8 & 1.3 & 2.1 & 2.0 \\ \hline
		\end{tabular} &
        \small
		\begin{tabular}{|c|c|c|c|c|c|}
			\hline
			$N_x$ & 10 & 20 & 40 & 80 & 160 \\ \hline
			Error & 1.9e-2 & 2.6e-2 & 2.0e-2 & 1.9e-2 & 7.4e-3 \\ \hline
			EoC(E) &  -- & -0.4 & 0.3 & 0.09 & 1.4\\ \hline
			$\eps_0$ & 2.6e-1 & 3.6e-1 & 2.7e-1 & 1.5e-1 & 3.4e-2 \\ \hline
			EoC($\eps_0$) & -- & -0.5 & 0.4 & 0.9 & 2 \\ \hline
		\end{tabular}
	\end{tabular}
	\caption{
    Convergence analysis of the discrete hyperbolized Cahn-Hilliard equation \eqref{eq:CHHDIMEX} towards the discrete Cahn-Hilliard equation \eqref{eq:CH_ht} as a function of $\eps$.
    Top row: Convergence of $\cc$ to $c$, bottom row: convergence of $\pp$ to $D_{\diamond} c$.
    The parameters $k_2$ and $k_3$ are set to one. Time integration is performed using an ARS-222 scheme, with the number of time steps $N_T = N_x$. The order of accuracy is set to $p=1$ for both computations.
    The tables correspond to the figures in the top row; they show the baseline error (the error of the SBP operator for the Cahn-Hilliard equation) and then the interpolated $\eps_0$ that is needed to obtain this baseline error using the CHH approximation.
	EoC denotes experimental order of convergence, EoC(E) is the experimental order of convergence of the error, and EoC($\eps_0$) that of $\eps_0$.
	}\label{fig:dg_chh_err_p1}
\end{figure}

\begin{figure}
	\begin{tabular}{cc}
        \includegraphics[scale=0.7]{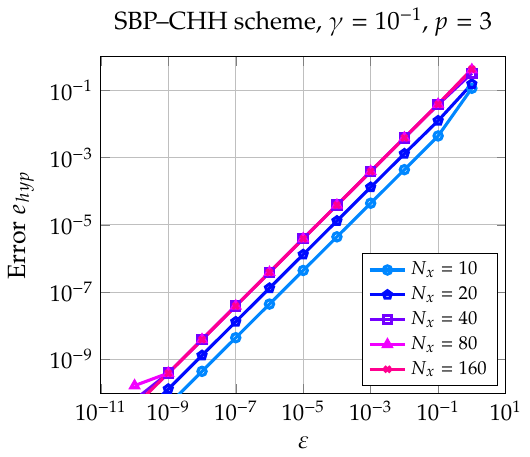} &
		\includegraphics[scale=0.7]{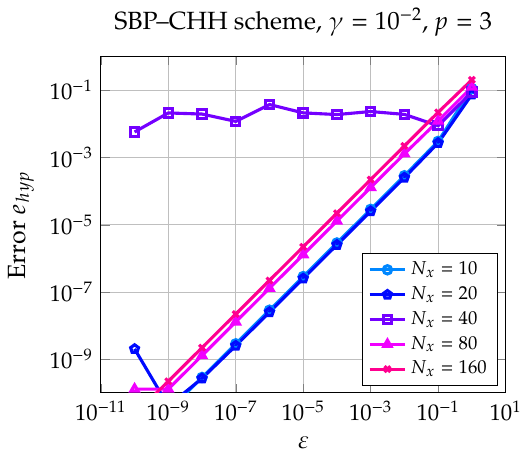} \\
        \includegraphics[scale=0.7]{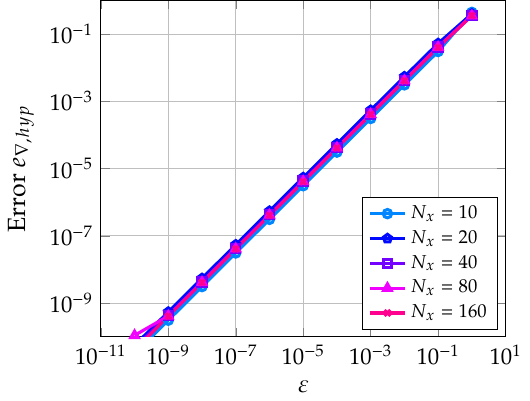} &
		\includegraphics[scale=0.7]{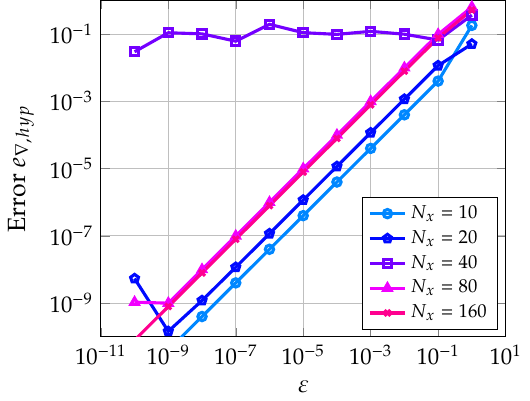} \\

        \small
		\begin{tabular}{|c|c|c|c|c|c|}
			\hline
			$N_x$ & 10 & 20 & 40 & 80 & 160 \\ \hline
			Error & 2.5e-1 & 4.4e-2 & 2.1e-2 & 1.6e-3 & 1.1e-4			 \\ \hline
			EoC(E) & -- & 2.5 & 1.0 & 3.8 & 3.9 \\ \hline
			$\eps_0$ & 2.1 & 3.0e-1 & 5.6e-2 & 4.0e-3 & 2.7e-4 \\ \hline
			EoC($\eps_0$) & -- & 2.8 & 2.4 & 3.8 & 3.9 \\ \hline
		\end{tabular} &
        \small
		\begin{tabular}{|c|c|c|c|c|c|}
			\hline
			$N_x$ & 10 & 20 & 40 & 80 & 160 \\ \hline
			Error & 1.4e-2 & 2.3e-2 & 2.1e-2 & 2.9e-2 & 5.1e-3  \\ \hline
			EoC(E) & -- & -0.8 & 0.1 & -0.4 & 2.5 \\ \hline
			$\eps_0$ & 1.9e-1 & 3.5e-1 & 2.3e-1 & 2.3e-1 & 2.3e-2 \\ \hline
			EoC($\eps_0$) & -- & -0.9 & 0.6 & --  & 3.3
			\\ \hline
		\end{tabular}
	\end{tabular}
	\caption{
    Convergence analysis of the discrete hyperbolized Cahn-Hilliard equation \eqref{eq:CHHDIMEX} towards the discrete Cahn-Hilliard equation \eqref{eq:CH_ht} as a function of $\eps$.
    Top row: Convergence of $\cc$ to $c$, bottom row: convergence of $\pp$ to $D_{\diamond} c$.
		The parameters $k_2$ and $k_3$ are set to one.
		Time integration is performed using an ARS-222 scheme, with the number of time steps $N_T = N_x$.
		The order of accuracy is set to $p=3$ for both computations.
		The tables correspond to the figures above; they show the baseline error (the error of the SBP operator for the Cahn-Hilliard equation) and then the interpolated $\eps_0$ that is needed to obtain this baseline error using the CHH approximation.
		EoC denotes experimental order of convergence, EoC(E) is the experimental order of convergence of the error, and EoC($\eps_0$) that of $\eps_0$.
	}\label{fig:dg_chh_err_p3}
\end{figure}

For the convergence analysis in Fig.~\ref{fig:dg_err_exact}, the definition of $D_{\circ}$ and $D_{\diamond}$ was irrelevant, as their product commutes. For the hyperbolization, however, it makes a difference. In the following, we define $D_{\circ} = D_+$ and $D_{\diamond} = D_-$.
In Figs.~\ref{fig:dg_chh_err_p1}--\ref{fig:dg_chh_err_p3}, we pick out $p= 1$ and $p= 3$, respectively, and approximate the numerical Cahn-Hilliard solution through the hyperbolized variant \eqref{eq:CHHDIMEX}.
The $\kappa$-parameters, see Thm.~\ref{thm:apimex}, are chosen as $\kappa_1 = \frac{\kappa_2}{\gamma} = \kappa_3 = \eps$, i.e., $k_2 = k_3 = 1$.
In all numerical experiments, unless otherwise stated, the initial conditions are set to be well-prepared. This means that, in line with Thm.~\ref{thm:apimex}, we set
\begin{align}
    \label{eq:discretewellpreparedics}
    \begin{split}
    \cc(\xx,t=0) &= \varphi(\xx,t=0) = c(\xx,t=0), \\
    \qq(\xx,t=0) = 0, \qquad w(\xx,t=0) &= 0, \qquad \pp(\xx,t=0) = D_{\diamond} \varphi(\xx,t=0).
    \end{split}
\end{align}

On the $x$-axis of Figs.~\ref{fig:dg_chh_err_p1}--\ref{fig:dg_chh_err_p3}, we report the values of $\eps$, while on the $y$-axis, we plot the error of the CHH equation with respect to the discretized Cahn-Hilliard equation on the same grid.
In the top row, we show the error of $\cc$ against $c$, while in the second row, we show the error of $\pp$ against $D_{\diamond} c \approx \partial_x c$.
As expected, the $\gamma = 10^{-1}$ case behaves more smoothly, which is due to the nature of the solution, which becomes discontinuous as $\gamma \rightarrow 0$, so the gradient is sharper for smaller $\gamma$.
This can be addressed by increasing the mesh resolution, which can also be seen from the figures.
Second, we see that for smaller $\eps$, the solutions converge towards each other.
What is very surprising is that the convergence results of both $\cc$ to $c$ and $\pp$ to $D_{\diamond} c$ look very similar.
This is a potential advantage for this approach, as $\pp$ could, e.g., be used for local postprocessing of the solution to increase the overall accuracy.
For $N_x = 40$ and $\gamma = 0.01$, in both cases $p=1$ and $p=3$, the convergence stalls a bit below the resolution error of the original method. We believe that this is due to the fact that the sharp gradient of the solution is resolved by too few points.
For $N_x = 20$, $\gamma = 0.01$ and very small $\eps$, we see some numerical instabilities that are related to numerical cancellation errors in terms such as $\frac{\cc - \varphi}{\eps}$.

Figs.~\ref{fig:dg_chh_err_p1}--\ref{fig:dg_chh_err_p3} show the error between the \emph{discretized} hyperbolized variant depending on $\eps$ and the \emph{discretized} Cahn-Hilliard equation.
For practitioners, the relevant error is between the discretized hyperbolized solution and the \emph{exact} Cahn-Hilliard solution.
Hence, in Figs.~\ref{fig:dg_chh_err_p1}--\ref{fig:dg_chh_err_p3} we also indicate which $\eps_0$ was necessary to produce CHH approximations that are as close to the discrete CH solution as the discrete CH solution is to the exact one.
This $\eps_0$ therefore balances the error contributions from discretization and hyperbolization; it has been determined through a linear interpolation based on the existing data points.
Experimental orders of convergence are reported for this $\eps_0$ and shown in comparison to the experimental orders of convergence of the SBP error.
It can be seen that these rates agree rather well, at least if there is sufficient resolution. This is not surprising, as the CHH solution converges with $\O(\eps)$ towards the CH solution.

\paragraph{Evolution of the energy}

We proved that both energies $\mathcal{E}_d$ and $\Ehd$ are decaying functions of time, see Lemma~\ref{la:energydecaych} and Lemma~\ref{la:energydecaychh}, at least if no time integration scheme has been applied.
In Fig.~\ref{fig:energy}, we plot the Cahn-Hilliard energy $\mathcal{E}_d$ and the hyperbolized energy $\Ehd$ for various values of $\eps$, as functions of time. The parameters were chosen in the same way as above, with $p = 3$ and $N_x = 80$.
It can be seen that all energies are decaying, as the theory predicts, even in this discretized setting. Furthermore, for $\eps \rightarrow 0$, it is obvious that the energies converge towards each other.
From $\eps \approx 10^{-2}$, the hyperbolized energy is virtually indistinguishable from the Cahn-Hilliard energy, again supporting the theory.

\begin{figure}
    \includegraphics{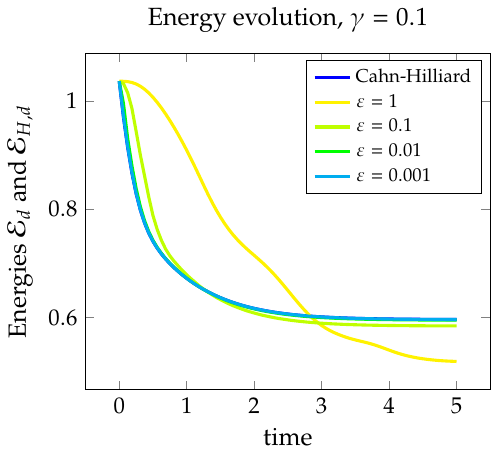}
	\includegraphics{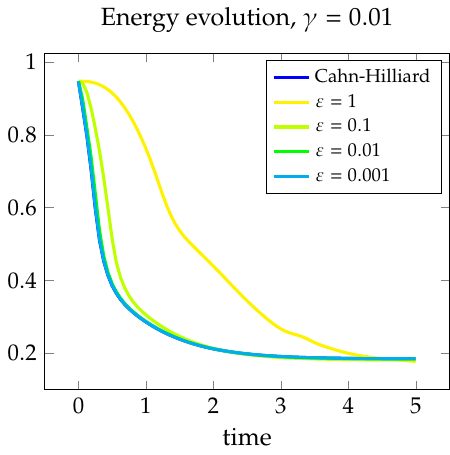}

	\caption{
		Evolution of the discrete energies $\E_{d}$ for the Cahn-Hilliard equation and $\Ehd$ for the hyperbolized variant for $\gamma = 0.1$ (left) and $\gamma = 0.01$ (right).  The parameters $k_2$ and $k_3$ are set to one. Time integration is performed using an ARS-222 scheme, with the number of time steps $N_T=N_x = 80$. The order of accuracy is $p=3$.
        }
	\label{fig:energy}
\end{figure}

\paragraph{Influence of $\kappa$}
Thm.~\ref{thm:aphyp} gives us some flexibility in the choice of the $\kappa$-parameters.
While $\kappa_1$ has to be fixed to $\eps$, $\kappa_2$ and $\kappa_3$ can be chosen as $\gamma \eps^{k_2}$ and $\eps^{k_3}$, respectively, for positive values of $k_2, k_3 \in \N$.
In Fig.~\ref{fig:compare_params}, we explore whether a different choice of the values $k_2$ and $k_3$ has an influence on the quality of the CHH approximation.
As an example, we show this analysis for $p=1$ and $N_x=80$, but the same conclusion is true for other values of $p$ and $N_x$.
In particular, the curves for other $p$ look very similar to the one for $p=1$, which demonstrates that the error here is mostly determined by the choice of $\eps$, and not so much by the discretization parameters.
It can hence be seen that changing the parameters $\kappa$ has hardly any influence on the solution quality for this test case; the convergence curves are pretty much identical.
This is also true (not shown here) for convergence results of $\pp$ against $D_{\diamond}c$.
We will hence continue to use $k_2 = k_3 = 1$ for this test case.

\begin{figure}
    \centering
	\includegraphics[height=0.4\textwidth]{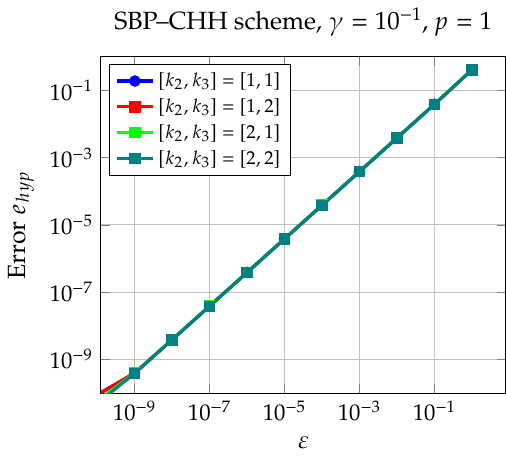}
	\includegraphics[height=0.4\textwidth]{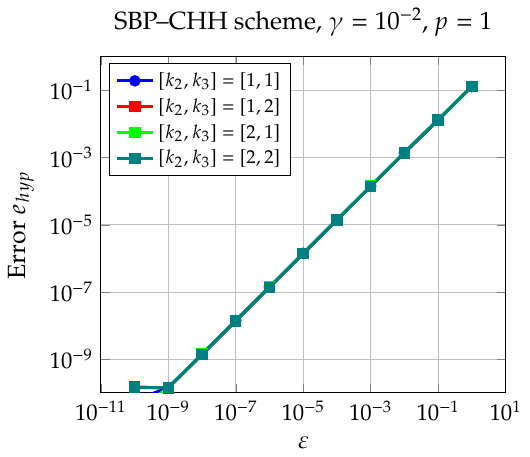}

	\caption{
    Convergence analysis of the discrete hyperbolized Cahn-Hilliard equation \eqref{eq:CHHDIMEX} towards the discrete Cahn-Hilliard equation \eqref{eq:CH_ht} as a function of $\eps$ for varying parameters $k_2$ and $k_3$.
		Time integration is performed using an ARS-222 scheme, with the number of time steps $N_T = N_x$. $N_x$ is set to 80; the order of accuracy is set to $p=1$ for both computations.
	}\label{fig:compare_params}
\end{figure}

\paragraph{Influence of well-prepared initial conditions}
As outlined in Thm.~\ref{thm:apimex}, the initial conditions need to be well-prepared, i.e., we set initial conditions as in Eq.~\eqref{eq:discretewellpreparedics}.
In this paragraph, we investigate what happens for initial conditions that are not well prepared, i.e., where we set all the hyperbolic initial conditions, except the one for $\cc$, to zero. The results of this investigation are shown in Fig.~\ref{fig:compare_wellprepared} for various values of $N_x$ and orders of accuracy of one and three, respectively.
At least for moderate values of $\eps$, the results do not differ significantly.
For small $\eps \ll 1$, however, the algorithm based on the ill-prepared initial data is significantly worse, to the extent that it diverges.
We expect that this is related to a loss of numerical accuracy. Using ill-prepared initial data means that terms such as $\frac{\cc - \varphi}{\kappa_1}$ in \eqref{eq:CHHDE3} can become extremely large.
This underlines the significance of using well-prepared conditions.
Please note that well-prepared initial conditions do not necessitate the solution of an additional problem; they can be easily computed through matrix-vector multiplication of the initial conditions.

\begin{figure}
	\centering
    \includegraphics{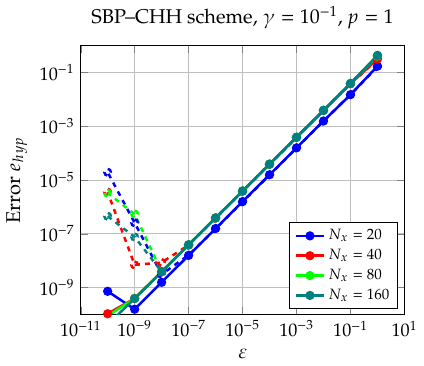}
	\includegraphics{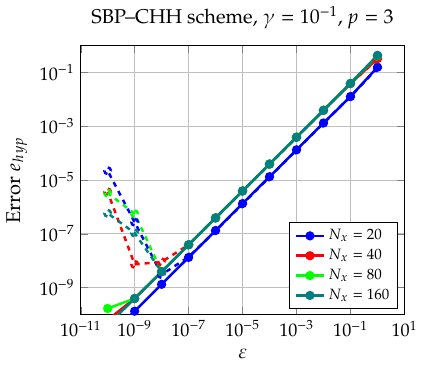}
	\includegraphics{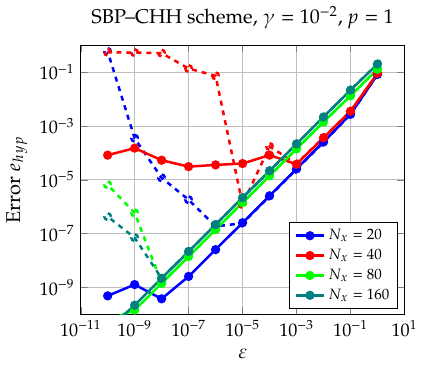}
	\includegraphics{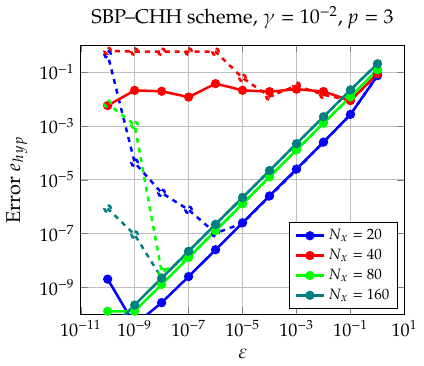}

	\caption{
    Convergence analysis of the discrete hyperbolized Cahn-Hilliard equation \eqref{eq:CHHDIMEX} towards the discrete Cahn-Hilliard equation \eqref{eq:CH_ht} as a function of $\eps$ under well-prepared and ill-prepared initial conditions. Dashed lines denote ill-prepared conditions. Time integration is performed using an ARS-222 scheme, with the number of time steps $N_T = N_x$. The order of accuracy is set to $p = 1$ (left) and $p=3$ (right).
	}\label{fig:compare_wellprepared}
\end{figure}

\FloatBarrier

\subsection{DG scheme}
In this section, we use the nodal LDG scheme as explained in \cite[Sec.~2.5]{RanochaConservative2021} for our investigations.
In contrast to the (periodic) finite-difference schemes, $D_+$ and $D_-$ do not commute for DG (except for $p = 1$, which reduces to a finite-difference scheme) and hence, it is interesting to analyze whether the definition of $D_{2}$ as $D_+ D_-$ or $D_- D_+$ makes a significant difference.
All the results in this section are computed with the same initial conditions as in \eqref{eq:chh_initial_cos} to allow for a direct comparison.
Also, the error is computed in the same way as in Sec.~\ref{sec:onedsbp}, of course with the mass matrix $M$ corresponding to the DG scheme. The nonlinear systems of algebraic equations are solved with a damped Newton method, also with parameters as given in Sec.~\ref{sec:onedsbp}.

Numerical results for the convergence of the scheme \eqref{eq:CH_h}, i.e., the direct discretization of the Cahn-Hilliard equation, are shown in Fig.~\ref{fig:dmdp}.
Solid lines correspond to $D_{2} = D_+D_-$ and dashed lines to $D_{2}=D_-D_+$. A difference is hardly visible.
While we checked that the results are really distinct, so we do not reproduce exactly the same results, the difference is so negligible that it cannot be seen on this plot.
We can hence conclude---and we see this in other numerical examples not shown here as well---that in practice, the definition of $D_{2}$ hardly plays a role, as long as it is a combination of $D_+$ and $D_-$ as outlined in the theory section.

\begin{figure}[t!]
    \centering
	\includegraphics[width=0.4\textwidth]{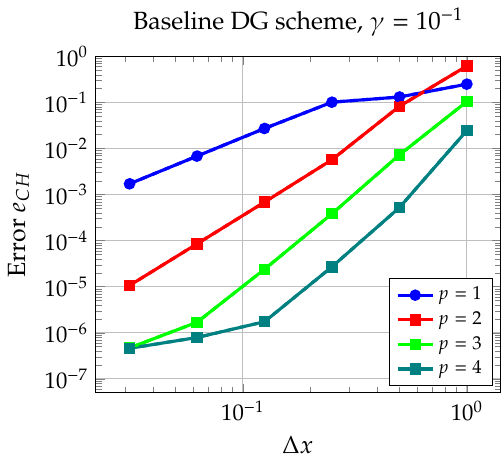}
	\includegraphics[width=0.4\textwidth]{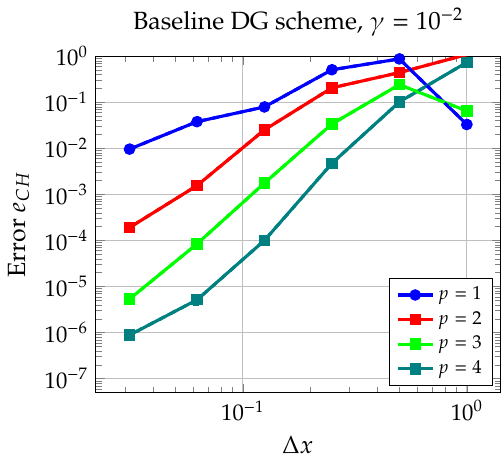}
	\caption{Convergence results of the original DG-SBP scheme using $D_{2} = D_+D_-$ (solid lines) and $D_{2}=D_-D_+$ (dashed lines) for various polynomial degrees and for $\gamma = 0.1$ (left) and $\gamma = 0.01$ (right).
	Time integration is performed using the ARS-222 scheme with $N_T = N_x$, and well-prepared initial conditions are used.
	The difference between solid and dashed lines is hardly visible.
	Results are not exactly the same, but they differ only very slightly.}\label{fig:dmdp}
\end{figure}

In Fig.~\ref{fig:dg_chh_err_p2}, we show the convergence results of the hyperbolized variant \eqref{eq:CHHDIMEX} against the baseline scheme \eqref{eq:CH_h}, as an example for polynomial degree $p=2$ for various values of $N_x$ as a function of $\eps$.
The convergence is much smoother than for the SBP finite-difference scheme, which is most likely due to the improved resolution of the DG scheme. We do not see any stalls of the convergence in $\eps$. As predicted by theory, the convergence is linear in $\eps$, i.e., for all values considered, we can see an experimental order of convergence of $\O(\eps)$.

\begin{figure}
        \centering
        \includegraphics[height=0.4\textwidth]{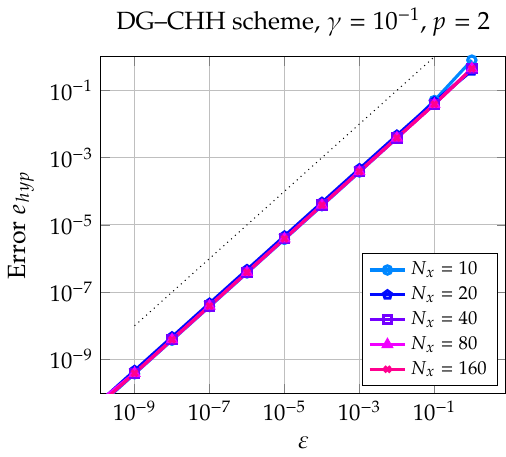}
		\includegraphics[height=0.4\textwidth]{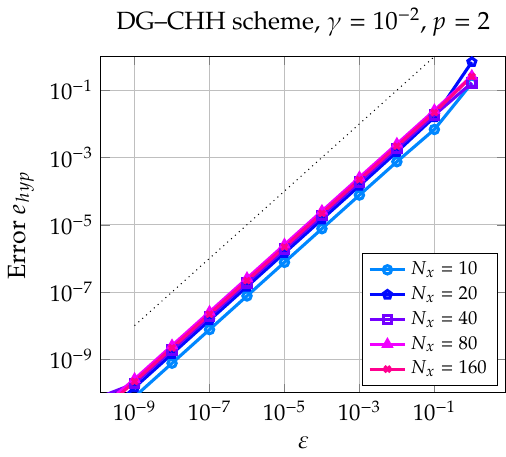}

	\caption{
    Convergence analysis of the discrete hyperbolized Cahn-Hilliard equation \eqref{eq:CHHDIMEX} with DG discretization towards the discrete Cahn-Hilliard equation \eqref{eq:CH_ht} as a function of $\eps$.
	The parameters $k_2$ and $k_3$ are set to one.
		Time integration is performed using an ARS-222 scheme, with the number of time steps $N_T = N_x$. The order of accuracy is set to $p=2$ for both computations. Dotted lines indicate first-order convergence in $\eps$.
	}\label{fig:dg_chh_err_p2}
\end{figure}

The framework of \cite{RanochaConservative2021} can also be used to formulate the BR1-DG scheme \cite{bassi1997high} as an SBP operator.
We have performed the same numerical tests for the BR1 scheme (not shown here).
While the convergence of the BR1 scheme for the Cahn-Hilliard discretization is worse than for LDG---as expected---we hardly see any difference in the convergence of the hyperbolized BR1 variant to the BR1 scheme.
The framework proposed here can hence also be used in a broader context.

\FloatBarrier

% !TeX root = main.tex

\subsection{2D results}\label{subsec:2DRes}

\begin{figure}
	\centering
	\includegraphics[width=0.18\textwidth,height=0.18\textwidth]{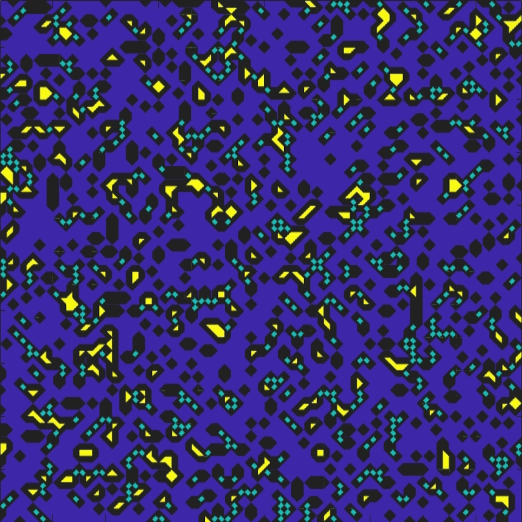}
	\includegraphics[width=0.18\textwidth,height=0.18\textwidth]{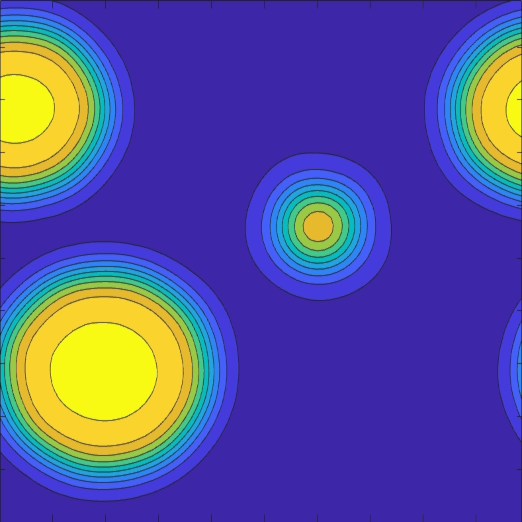}
	\includegraphics[width=0.18\textwidth,height=0.18\textwidth]{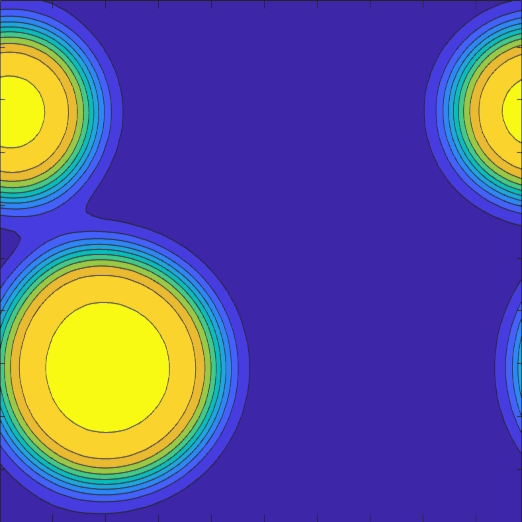}
	\includegraphics[width=0.18\textwidth,height=0.18\textwidth]{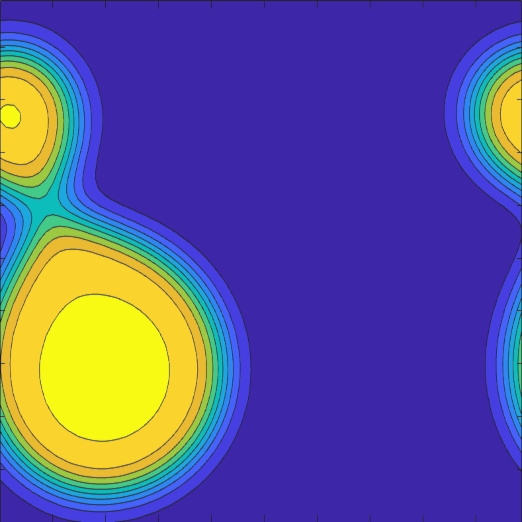}
	\includegraphics[width=0.18\textwidth,height=0.18\textwidth]{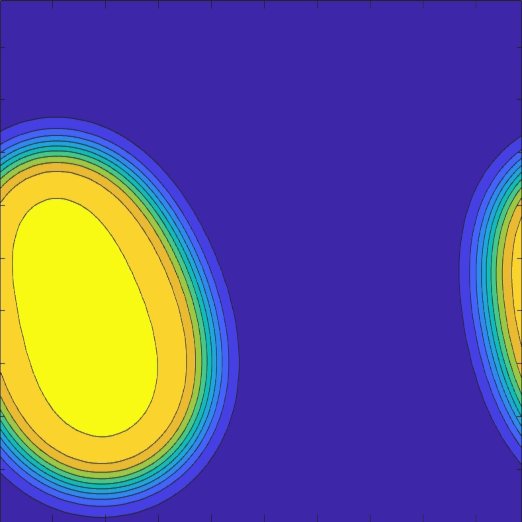} \\[0.5em]
	%%%
	\includegraphics[width=0.18\textwidth,height=0.18\textwidth]{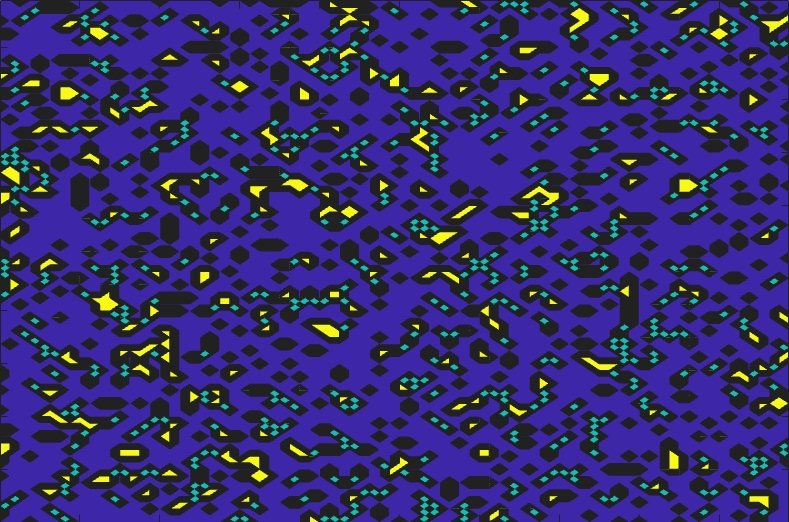}
	\includegraphics[width=0.18\textwidth,height=0.18\textwidth]{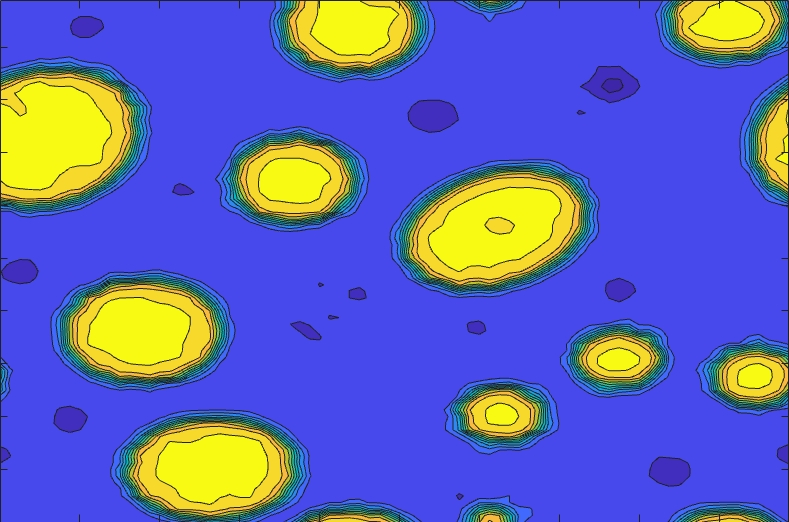}
	\includegraphics[width=0.18\textwidth,height=0.18\textwidth]{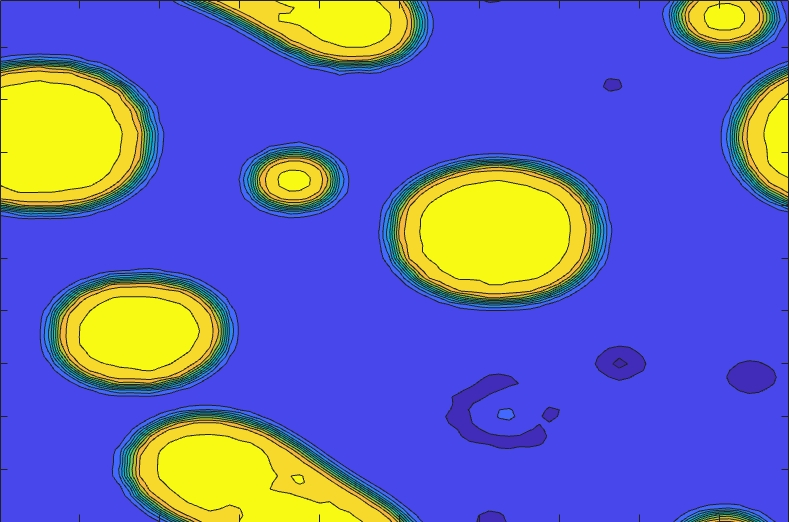}
	\includegraphics[width=0.18\textwidth,height=0.18\textwidth]{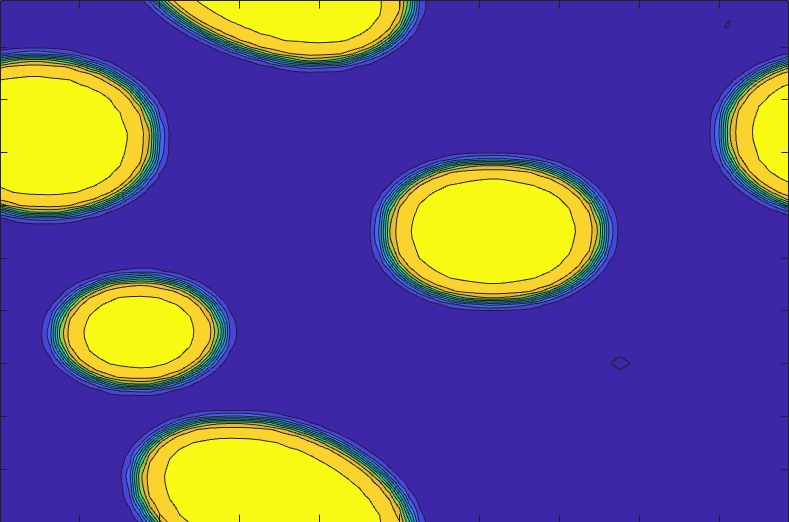}
	\includegraphics[width=0.18\textwidth,height=0.18\textwidth]{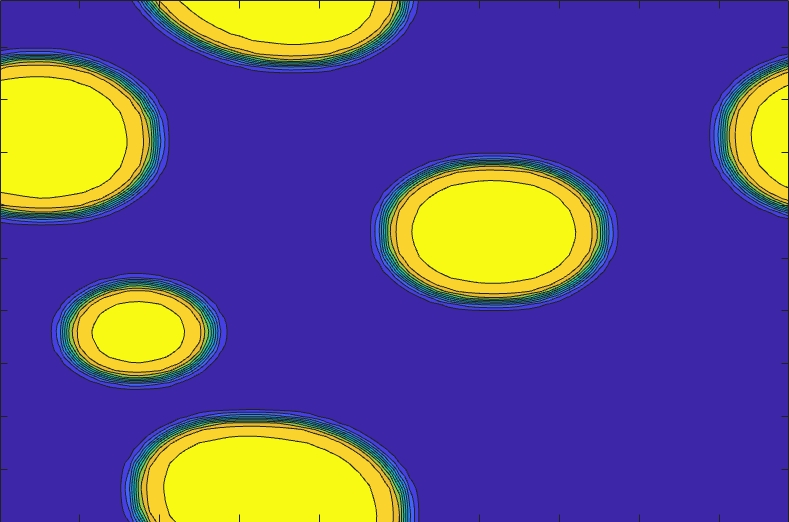}
	\caption{Two-dimensional results of the Cahn-Hilliard equation on the unit rectangle for $\gamma = 10^{-3}$ (top) and $\gamma = 10^{-4}$ (bottom) at times $t \in \left\{0, \frac 1 4, \frac 1 2, \frac 3 4, 1 \right\}$.
		The spatial resolution is $N_x = N_y = 80$, the computation is performed with $p=1$ and the ARS-222 time integration scheme.}
	\label{fig:CH2D}
\end{figure}

In this section, we consider the Cahn-Hilliard equation on a two-dimensional domain.
More precisely, we choose $\Omega = [0, 1]^2$. Again, periodic boundary conditions are used.
As the initial condition, we choose a random field consisting of approximately 75\% of the entries equal to -1, and 25\% equal to +1. More precisely,
\begin{align*}
	c(x,0) = \sign(4 \rand - 3),
\end{align*}
where $\rand$ generates uniformly distributed values\footnote{In MATLAB, we implement $\rand$ through the command ``rng(4711), rand($N_x$*$N_y$,1)'', with $N_x$ the spatial resolution in the $x$-direction and $N_y$ the spatial resolution in the $y$-direction.} in $[0, 1]$.
This produces initial conditions as shown in Fig.~\ref{fig:CH2D}, leftmost pictures.
The dynamics of the Cahn-Hilliard equation will merge these randomly generated droplets into larger circles that will eventually collapse into one large circle.
The solutions for $\gamma = 10^{-3}$ and $\gamma = 10^{-4}$ are shown in Fig.~\ref{fig:CH2D} for times $t \in \left\{0, \frac 1 4, \frac 1 2, \frac 3 4, 1 \right\}$. In this section, we will always work with $\Te = 1$.
As in the previous section, the error is computed at $\Te$, with the same definition as for the one-dimensional case.
Also in this section, for the resulting nonlinear systems of equations, we use a damped Newton procedure with a maximum of 20 steps, and a reduced absolute and relative tolerance of $10^{-8}$.

The numerical discretization in two dimensions is straightforward; we use the SBP finite-difference schemes.
We work with a tensor-product grid consisting of $N_x$ degrees of freedom in the $x$-direction and $N_y$ degrees of freedom in the $y$-direction.
Due to this tensor-product structure, we can then define $x$-derivatives as
\begin{align*}
	D_{\circ,x} := \Id_{N_y} \otimes D_{\circ},
\end{align*}
with $D_{\circ}$ the one-dimensional derivative operator $D_+$ or $D_-$ on a grid with $N_x$ degrees of freedom, and $\Id_{N_y}$ the $(N_y \times N_y)$ identity matrix.
Similarly, the $y$-derivatives are defined as
\begin{align*}
	D_{\circ,y} := D_{\circ} \otimes \Id_{N_x}.
\end{align*}
One can show that these operators are also SBP operators.
If we then define the approximation of the gradient through $D_{\circ,\nabla} := \begin{pmatrix} D_{\circ, x} \\ D_{\circ, y} \end{pmatrix}$, we can use exactly the same scheme as in \eqref{eq:CHHDIMEX}.
In particular, the combination of $+$ and $-$ operators is the same as in the one-dimensional case.
Moreover, energy-stability can be shown in exactly the same fashion.
In this section, we first work with the finite-difference SBP operators.

Based on the analytical results in Thm.~\ref{thm:error_estimate}, we decided to scale $\kappa_2 = \gamma \eps^{k_2}$.
A more naive choice would obviously be not to include $\gamma$ in $\kappa_2$---all proofs would go through with this choice as well.
To observe the influence of this choice, in Fig.~\ref{fig:CHH2DConvg}, we show convergence results of the hyperbolized equation to the Cahn-Hilliard equation on a grid with $N_x = N_y = 80$, corresponding to the solutions shown in Fig.~\ref{fig:CH2D}, for various values of $\eps$, $\gamma$ and $p$; the results have all been computed using well-prepared initial conditions.
We distinguish between the naive choice of $\kappa_2$ without $\gamma$ (dashed lines) and the choice of $\kappa_2 = \gamma \eps^{k_2}$ as suggested in this work (solid lines).
Furthermore, we test $[k_2,k_3] = [1,1]$ (blue) versus $[k_2,k_3] = [2,1]$ (red). Note that the latter variant leads to an equilibration (in the orders of $\eps$) of the wave speeds in \eqref{eq:wavespeeds}.

First, it can be seen that in each case, for $\eps \rightarrow 0$, the solutions converge to each other.
For the naive choice of $\kappa_2$, the choice of the parameter $k_2$ matters significantly; the equilibrated variant with $k_2 = 2$ behaves significantly better, where the effect is more pronounced for smaller values of $\gamma$.

For the $\gamma$-scaled choice of $\kappa_2$, we hardly see an influence of the value of $k_2$.
In particular, convergence ``starts'' at significantly lower values of $\eps$ compared to the choice $\kappa_2 = \eps$.
These results are nicely in line with the results of Thm.~\ref{thm:error_estimate}, as Eq.~\eqref{eq:the error estimate} suggests---upon ignoring that the constants might depend on $\gamma$ as well---that the error scales as $\kappa_1 + \frac{\kappa_2} \gamma + \kappa_3$.
Thus, an optimal scaling is reached if all terms behave in the same way, which means that $\eps = \kappa_1 = \frac{\kappa_2}{\gamma} = \kappa_3$ should hold. If one chooses one of these contributions smaller than the others, this will hardly lead to an improvement, as the error is then simply dominated by the other two components.
We did not plot the error of $\pp$ against the discrete gradient of $\cc$, as the curves again look very similar to those for $e_{hyp}$, which is in very good agreement with the one-dimensional results.

\begin{figure}
    \centering
	\includegraphics[height=0.33\textwidth]{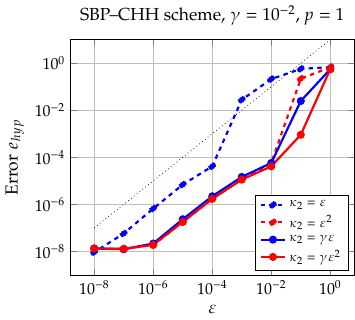}
    \includegraphics[height=0.33\textwidth]{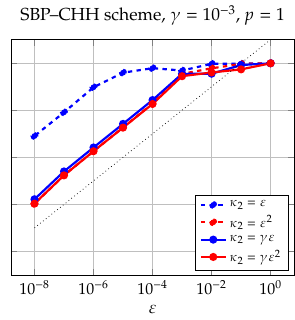}
    \includegraphics[height=0.33\textwidth]{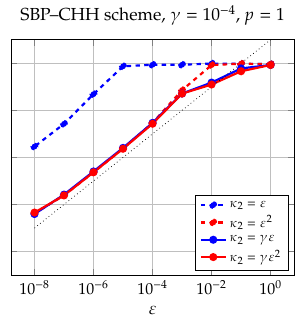} \\
    \includegraphics[height=0.33\textwidth]{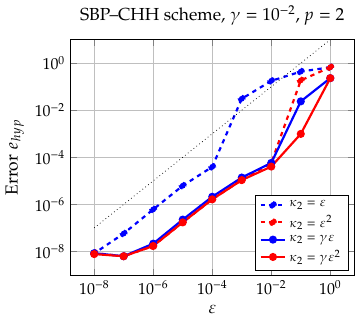}
    \includegraphics[height=0.33\textwidth]{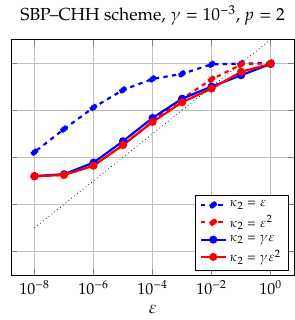}
    \includegraphics[height=0.33\textwidth]{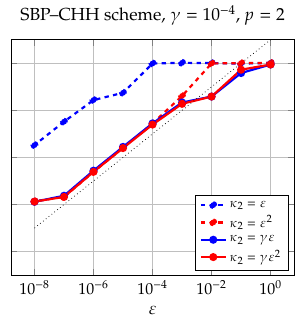}
	\caption{
    Convergence analysis of the discrete hyperbolized two-dimensional Cahn-Hilliard equation towards the discrete two-dimensional Cahn-Hilliard equation as a function of $\eps$.
	Time integration is performed using an ARS-222 scheme, with $N_T = N_x$ time steps.
	We choose the spatial resolution to be $N_x = N_y = 80$. $\gamma$ varies from $10^{-2}$ to $10^{-4}$.
	Solid lines correspond to the usual scaling of $\kappa_2 = \gamma \eps^{k_2}$. Dashed lines correspond to the naive scaling $\kappa_2 = \eps^{k_2}$. Blue denotes $[k_2, k_3] = [1,1]$, while red denotes the choice $[k_2, k_3] = [2,1]$.
	The dotted black line indicates first order of convergence in $\eps$.
	}\label{fig:CHH2DConvg}
\end{figure}

In Fig.~\ref{fig:2Denergy}, we show the evolution of the discrete energies $\E_{d}$ and $\Ehd$, respectively, for various values of $\eps$.
We choose the order of accuracy $p= 1$, and a resolution of $N_x = N_y = 80$ (other choices of parameters yield similar results).
The parameter $\gamma$ is set to $10^{-3}$. Again, we compare $[k_2,k_3]=[1,1]$ to $[k_2,k_3] = [2,1]$ (dashed).
The initial conditions are highly irregular, which imposes a very high energy at time $t = 0$.
In fact, the energy looks nearly discontinuous for small $t$, which is why we zoom in on the energy in the right part of the plot. We can repeat the conclusion from the one-dimensional case that all energies are decaying, and that for $\eps \rightarrow 0$, they converge to each other.
Again, we see hardly any influence; only for large values of $\eps$, the $k_2 = 2$ case behaves slightly better (obviously not for $\eps = 1$, where both values yield exactly the same results, as $1 = 1^{k_2}$).

\begin{figure}
    \centering
	\includegraphics[width=0.4\textwidth]{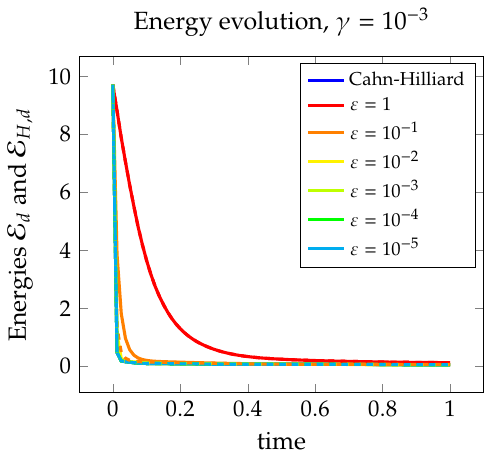}
	\includegraphics[width=0.4\textwidth]{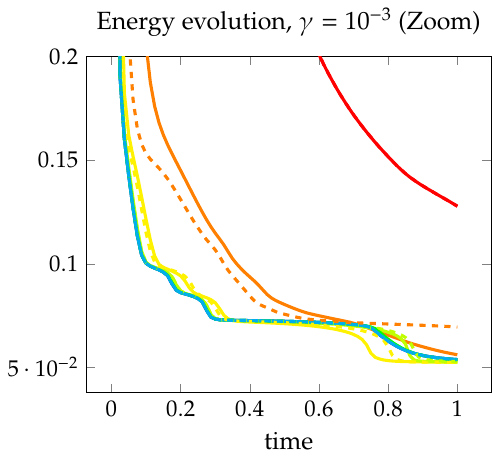}
	\caption{
		Evolution of the discrete energies $\E_{d}$ for the two-dimensional Cahn-Hilliard equation and $\Ehd$ for the hyperbolized variant for $\gamma = 10^{-3}$.
			Time integration is performed using an ARS-222 scheme, with the number of time steps $N_T=N_x = N_y = 80$. The order of accuracy is $p=1$.
        Solid lines correspond to the choice of the parameters $[k_2, k_3] = [1,1]$, while dashed lines correspond to $[k_2, k_3] = [2, 1]$.
        Note that the initial condition is so irregular that it imposes a very high energy in the beginning.
		For small $\eps$, the energy of the hyperbolized version is indistinguishable from the original Cahn-Hilliard energy in this plot.
        }
	\label{fig:2Denergy}
\end{figure}

\paragraph{Influence of $k_2$ and $k_3$}

As in the one-dimensional section, we aim to investigate the influence of the parameters $\kappa$ from Thm.~\ref{thm:aphyp}.
This analysis is done using an SBP finite-difference scheme with $N_x = N_y = 60$ and order of accuracy $p=2$, and an LDG scheme with $N_x = N_y = 15$ and $p=3$.
Note that $\ndof$ is the same for both computations. The Cahn-Hilliard parameter is set to $\gamma = 10^{-3}$. The remaining parameters are chosen as before.
The convergence results can be found in Fig.~\ref{fig:2DDGKappas}.
As expected, it can be seen that the choice of $\kappa$ has a rather small influence on the solution quality.
Among the parameters investigated here, $[k_2,k_3]=[2,2]$ and $[k_2,k_3] = [1,2]$ behave the best (and very similar), better than $[k_2,k_3]=[1,1]$ and $[k_2,k_3]=[2,1]$, which again behave nearly the same way.
However, all choices converge, and the difference is not extremely pronounced.
Please also note that the algebraic systems of equations to be solved get stiffer with larger $k_2, k_3$ and smaller $\eps$.

\begin{figure}
    \centering
    \includegraphics[width=0.4\textwidth]{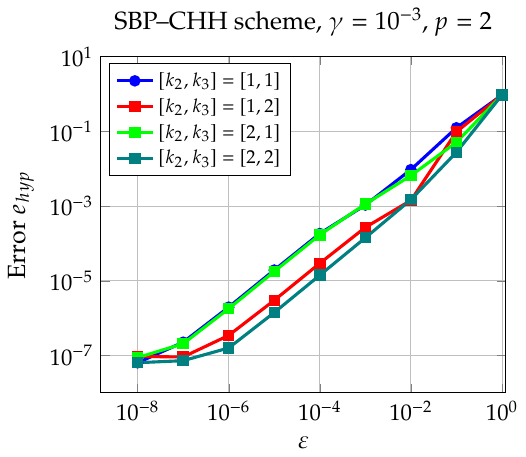}
    \includegraphics[width=0.4\textwidth]{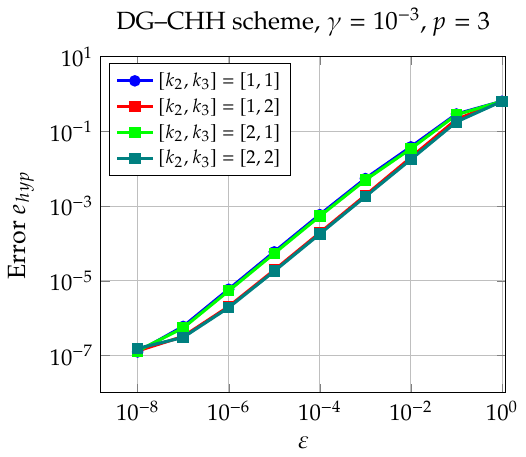}
    \caption{Convergence analysis of the discrete hyperbolized Cahn-Hilliard equation \eqref{eq:CHHDIMEX} towards the discrete Cahn-Hilliard equation \eqref{eq:CH_ht} as a function of $\eps$ for varying parameters $k_2$ and $k_3$.
	Time integration is performed using an ARS-222 scheme, with the number of time steps $N_T = N_x$. $N_x$ and $N_y$ are set to 60 (left) and 15 (right).
	The left plot uses an SBP scheme with order of accuracy 2, the right plot an LDG scheme with polynomial degree $p= 3$. Note that in both cases, $\ndof$ is identical.}
    \label{fig:2DDGKappas}
\end{figure}

\section{Summary and conclusions}
\label{sec:conout}
In this work, we analytically investigated a hyperbolization of the Cahn-Hilliard equation proposed in \cite{dhaouadi2025first}, proved a~priori error estimates, developed provably energy-stable discretization schemes for both the Cahn-Hilliard equation and its hyperbolized variant, and showed numerical results.
In particular, we have investigated how to choose the parameters of the hyperbolization optimally.

Future work will concentrate on several further research questions. First, there exists a structurally different quasi-hyperbolization of the Cahn-Hilliard equation, given in \cite{keim2024note}, which has fewer unknowns but is not fully hyperbolic.
A comparison of both hyperbolizations in terms of a~priori error estimates, convergence, etc., is of key interest to evaluate the area of application of both variants.
Furthermore, it might be of interest to revisit the splitting of \cite{dhaouadi2025first} in the light of these results, to see whether one can reduce the number of unknowns.

We purposely did not report the timing results of our algorithms in this work.
The hyperbolized variant is typically much slower, which has a multitude of reasons.
First, the hyperbolized equation consists of significantly more unknowns; second, the nonlinear algebraic systems of equations become extremely stiff for $\eps \rightarrow 0$.
While the first issue is related to the hyperbolization itself, and can only be tackled by reconsidering the splitting, the second issue needs to be treated by carefully devising algebraic solvers capable of handling the stiffness.
In particular, a well-chosen combination of preconditioners and Krylov subspace methods needs to be developed and analyzed.

As a further area of research, we consider the application-relevant case where the Cahn-Hilliard equation is coupled to other equations, such as the Navier-Stokes equations \cite{AbelsGarckeGruen2012}.
In this context, hyperbolization seems an even more natural choice, in particular, when combined with the compressible Navier-Stokes equations, which are largely dominated by their hyperbolic parts and where there is a long tradition of using hyperbolic solvers.

\appendix

\section*{Acknowledgments}

JG was supported by the German Research Foundation (DFG) within the projects No.~525866748 and No.~525877563 of the Priority Program SPP 2410 Hyperbolic Balance Laws in Fluid Mechanics: Complexity, Scales, Randomness (CoScaRa).
HR was supported by the Deutsche Forschungsgemeinschaft
(DFG, German Research Foundation, project numbers 513301895 and 528753982
as well as within the DFG priority program SPP~2410 with project number 526031774).
JS was supported by BOF funding from UHasselt.

\printbibliography

\clearpage

\end{document}